\documentclass[11pt]{amsart}
\usepackage{amsmath, amssymb,mathabx}
\usepackage[colorlinks=true,linkcolor=blue,urlcolor=blue,citecolor=blue, pagebackref]{hyperref}
\usepackage[alphabetic]{amsrefs}
\usepackage{amsfonts, epsfig,color,wrapfig}
\usepackage{ color}
\usepackage[all]{xy}
\usepackage{graphicx}
\numberwithin{equation}{section}

\newcommand\op{\operatorname}

\newcommand\frg{\mathfrak{g}}
\newcommand\frh{\mathfrak{h}}

\newcommand\mf{\mathcal{F}}

\newcommand\Pic{\operatorname{Pic}}

\newcommand\Fl{\operatorname{Fl}}

\newcommand\tensor{\otimes}
\newcommand\ml{\mathcal{L}}

\newcommand\GL{\operatorname{GL}}

\newcommand\mh{\mathcal{H}}

\newcommand\mc{\mathcal{C}}

\newcommand\Hom{\operatorname{Hom}}

\newcommand\SL{\operatorname{SL}}

\newcommand{\leto}[1]{\stackrel{#1}{\to}}

\newcommand\frl{\mathfrak{l}}

\newcommand\Ind{\op{Ind}}

\newtheorem{theorem}{Theorem}[section]

\newtheorem{remark}[theorem]{ Remark}

\newtheorem{corollary}[theorem]{Corollary}

\newtheorem{proposition}[theorem]{Proposition}

\newtheorem{lemma}[theorem]{Lemma}

\newtheorem{definition/lemma}[theorem]{Definition/Lemma}
\newtheorem{defi}[theorem]{Definition}

\begin{document}
\title[Extremal rays of eigencones]{Extremal rays in the Hermitian eigenvalue problem for arbitrary types}
\author{Prakash Belkale and Joshua Kiers}
\begin{abstract}
The Hermitian eigenvalue problem asks for the possible eigenvalues of a sum of Hermitian matrices given the eigenvalues of the summands. This is a problem about  the Lie algebra of the maximal compact subgroup of $G=\op{SL}(n)$ . There is a polyhedral cone (the ``eigencone'') determining the possible answers to the problem. These eigencones can be defined for arbitrary semisimple groups $G$, and also control the (suitably stabilized) problem of existence of non-zero invariants in tensor products of irreducible representations of $G$.

We give a description of the extremal rays of the eigencones for arbitrary semisimple groups $G$ by first observing that extremal rays lie on regular facets, and then classifying extremal rays on an arbitrary regular face.  Explicit formulas are given for some extremal rays, which have an explicit geometric meaning as cycle classes of interesting loci, on an arbitrary regular face, and the remaining extremal rays on that face are understood by a geometric process we introduce, and explicate numerically, called induction from Levi subgroups. Several numerical examples are given. The main results, and methods, of this paper generalize  \cite{BHermit} which handled the case of $G=\SL(n)$.

\end{abstract}
\maketitle
\section{Introduction}
The Hermitian eigenvalue problem asks for the possible eigenvalues of a sum of Hermitian matrices given the eigenvalue of the summands (see e.g., \cite{brion,kumar} for recent surveys). In its Lie theoretic formulation, this is a problem about the Lie algebra of the maximal compact subgroup $K=\op{SU}(n)$ of $G=\op{SL}(n)$. There is a polyhedral cone, $\Gamma(s,K)$, controlling the possible eigenvalues, and the problem can be generalised to an arbitrary semisimple group $G$. The corresponding polyhedral cones are called eigencones, and are  important objects in representation theory which also control saturated versions of the tensor decomposition problem. In this paper we give an inductive determination of extremal rays of these eigencones.

Fix a Cartan decomposition of $G$. It is known that $\Gamma(s,K)\subseteq \frh_{+}^s$ is a polyhedral cone (here $\frh_+$ is the positive Weyl chamber of $G$, see below).  It is known that  $\Gamma(s,K)$ is cut out inside $\frh_{+}^s$
by a system of inequalities controlled by the Schubert calculus of homogenous spaces $G/P$ where
$P$ runs through all standard maximal parabolics of $G$. The regular faces (see definition \ref{rf}) of the polyhedral cone $\Gamma(s,K)$ have also been determined, see \cite{Kly,BLocal,KTW,BeS,KLM,BK,R1,R2}, and the survey \cite{kumar}. The results cited above on regular faces do not include  explicit (i.e., with formulas) procedures of manufacturing elements in $\Gamma(s,K)$ on faces, and do not give information about the extremal rays of $\Gamma(s,K)$, or the structure of regular faces beyond their dimension.

We start by observing that all extremal rays lie on regular facets (Definition \ref{rf}, and Lemma \ref{extreme}). We then give  formulas for some extremal rays, which have an explicit geometric meaning as cycle classes of interesting loci, on an arbitrary regular face (Theorems \ref{Tone} and \ref{timeticks}), and show that the remaining rays on that face can be explicitly understood by a geometric process we introduce, called induction from Levi subgroups (Theorem \ref{IndT}). Theorems \ref{timeticks} and \ref{IndT} produce explicit formulas, featuring intersection numbers, which  generate the regular faces of $\Gamma(s,K)$.

The main results, and methods of this paper generalize  \cite{BHermit} which handled the case of $G=\SL(n)$ (and facets). While the results of \cite{BHermit} used the classical geometry of flag varieties in type A, we rely here on more general Lie theoretic methods. In particular, the ramification results of \cite{BKR} play a crucial role in the construction of the induction operation (this phenomenon does not appear for type A maximal parabolics). The generalized Fulton conjecture proved in \cite{BKR} plays a key role in the proof of Theorem \ref{Tone}. We build upon earlier work on the minimal set of inequalities defining $\Gamma(s,K)$ \cite{BLocal,BK,R1}, and the attendant ramification theory in enumerative problems \cite{belkaleIMRN,BK,BKR}. The induction operation is inspired by \cite{R1}, the basic divisors $D(j,v)$ (\eqref {defDD} below), which are shown here to give extremal rays, appear implicitly in \cite{BKR}. These basic divisors can also be seen to correspond to  $E_j$ considered in \cite[Section 4.1]{R1} for an optimal choice of $X^o$ (as in loc. cit.), given the results of \cite{BKR}.

\subsection{The eigencones}
Let $G$ be a semisimple, connected complex algebraic group, with a Borel subgroup $B$ and a maximal torus $T\subset B$. Let $W=W_G=N_G(T)/T$ be the associated Weyl group, where $N_G(T)$ is the normalizer  of $T$ in $G$.  Our choice of $B$ and $T$ fixes a Cartan decomposition of the Lie algebra $\frg$ of $G$. Let $\frh$ be the Lie algebra of $\frh$ and $\frh_{\Bbb{R}}$  the real vector space spanned by the co-roots of $G$. Let $K$ be the maximal compact subgroup of $G$ with Lie algebra $\mathfrak{k}$, chosen such that $i\frh_{\Bbb{R}}$ is the Lie algebra of a maximal torus of $K$.

Let  $\frh_{+}$ be   the positive Weyl chamber in $\frh_{\Bbb{R}}$. There is a bijection $C:\frh_+\leto{\sim} \mathfrak{k}/K$ where $K$ acts on $\mathfrak{k}$ by the natural action of a Lie group on its Lie algebra. Let $\Gamma(s,K)\subseteq \frh_{+}^s$ be the ``eigencone'', with $s\geq 3$:
\begin{equation}\label{kishori}
\Gamma(s,K)=\{(h_1,\dots,h_s)\mid \exists k_1,\dots,k_s\in \mathfrak{k},\  C(h_j)=\overline{k_j},\ j=1,\dots,s,\ \sum_{j=1}^s k_j=0 \}
\end{equation}
To state the minimal set of inequalities determining $\Gamma(s,K)$, we introduce some notation first.
\subsection{Notation}
Our choice of $T$ and $B$ fixes a Cartan decomposition of $\frg$. Let $R\subset \frh^*$ (resp $R^+,R^-$) be the set of roots (resp. positive roots, negative roots) of $G$. Let $\Delta=\{\alpha_1,\dots,\alpha_{r}\}\subset R^+$ be the set of simple roots. Let $\{x_1,\hdots,x_r\}$ be the basis of $\frh$ dual to $\Delta$; i.e., $\alpha_i(x_j) = \delta_{i,j}$. Let $\{\alpha_1^{\vee},\dots,\alpha_{r}^{\vee}\}\subset \frh$ be the set of simple coroots. Let $\omega_1,\dots,\omega_r\in \frh^*$ (dominant fundamental weights) be the basis dual to the simple coroots, so that $\omega_i(\alpha_j^{\vee})=\delta_{ij}$.

Let $\frh_\Bbb{Q}$ denote the $\Bbb{Q}$-span of the simple coroots, and set $\frh_{+,\Bbb{Q}} = \frh_+\cap \frh_\Bbb{Q}$. We have an isomorphism $\kappa:\frh^*_{\Bbb{Q}}\to \frh_{\Bbb{Q}}$ induced by the Killing form, where, with $\frh^*_\mathbb{Z}$ denoting the integer span of the fundamental weights, $\frh^*_\mathbb{Q}=\frh^*_\mathbb{Z}\otimes \mathbb{Q}$. The mapping takes the rational cone  generated by dominant fundamental weights $\frh^*_{+,\Bbb{Q}}$ to the rational Weyl chamber $\frh_{+,\Bbb{Q}}$. $\Gamma(s,K)\subseteq \frh_{+}^s$ is a rational polyhedral cone, and we denote by  $\Gamma_{\Bbb{Q}}(s,K)\subseteq \frh_{+,\Bbb{Q}}^s$ the corresponding rational cone.

Let $P$ be any standard parabolic of $G$ (not necessarily maximal) and $U=U_P$
be its unipotent radical. Let $L=L_P$ be the Levi subgroup of $P$, which has a Borel subgroup $B_L= B\cap L$. The Lie algebras of $G,B,T,P,U,L,B_L$ and $K$ are denoted by
$\frg,\mathfrak{b},\mathfrak{h},\mathfrak{p},\mathfrak{l}, \mathfrak{b}_L$ and $\mathfrak{k}$ respectively. Let $X(T)$ be the group of multiplicative characters $T\to \Bbb{C}^*$.
Let $R_{\mathfrak{l}}\subseteq R$ (resp $R^+_{\mathfrak{l}}, R^-_{\mathfrak{l}}$) be the subset of roots (resp. positive roots, negative roots) of $L$ and $\Delta(P)$ the set of simple roots in $R_{\mathfrak{l}}$.

The Weyl group of $P$, $W_P$, is by definition the Weyl group of $L$.
In any coset of $W/W_P$, there is a unique element $w$ of minimal length, and it satisfies $wB_Lw^{-1}\subseteq B$. Let $W^P\subseteq W$ be the set of minimal length representatives in the cosets of $W/W_P$.

Let  $w\in W^P$. Define the Schubert cell $C_w\subset G/P$ by
$C_w= BwP/P$. Let $X_w$ be the closure $\overline{C}_w\subseteq G/P$. Let $[X_w]\in H^{2\dim G/P- 2\ell(w)}(G/P,\Bbb{Z})$ be the cycle class of $X_w$.

\subsection{The system of inequalities determining $\Gamma(s,K)$}\label{regular1}
It is known (see \cite{BK}) that $(h_1,\hdots,h_s)\in\Gamma(s,K)$ if and only if, for every standard parabolic $P\subset G$ and $w_1,\hdots,w_s\in W^P$ such that
\begin{equation}\label{kishi}
[X_{w_1}]\odot_0 [X_{w_2}]\odot_0\dots\odot_0[X_{w_s}]=[X_e]\in H^*(G/P),
\end{equation}
the inequality
\begin{equation}\label{fox}
\sum_{j=1}^s \omega_k\left(w_j^{-1} h_j\right) \le 0
\end{equation}
holds, where $\omega_k$ is the fundamental weight corresponding to any simple root $\alpha_k\in \Delta\setminus \Delta(P)$. Here $\odot_0$ is the deformation of the cup product on $H^*(G/P,\Bbb{C})$ introduced in \cite{BK}.
\begin{defi}\label{rf}
A face of $\Gamma(s,K)\subseteq \frh_{+}^s$ is said to be regular if it is not contained  in one of the
Weyl chamber walls on any of the $s$ factors, i.e., the face is not contained in $\{(h_1,\dots,h_s)\mid \alpha_j(h_i)=0\}$ for some $(i,j)$.
\end{defi}
In fact, the inequalities arising only from maximal parabolics $P$ suffice to determine $\Gamma(s,K)$, and these are irredundant: each inequality \eqref{fox} corresponding to the above data with $P$ maximal determines a regular facet of $\Gamma(s,K)$  and all regular facets arise this way \cite{R1}. It will be shown that any extremal ray of $\Gamma(s,K)$ lies on  some such regular facet (Lemma \ref{extreme}). Therefore it suffices to determine the extremal rays of all regular facets. Here $\odot_0$ is the deformation of the usual cohomology product introduced in \cite{BK}.

Fix a (possibly non-maximal) parabolic $P$  in $G$, and $w_1,\dots,w_s\in W^P$ such that \eqref{kishi} holds. Define the face $\mf(\vec w,P)$ of $\Gamma(s,K)\subseteq \frh_{+}^s$ by
\begin{equation}\label{faceF}
\mf(\vec w,P)=\{(h_1,\dots,h_s)\in \Gamma(s,K)\mid \sum_{j=1}^s \omega_{{k}}(w_j ^{-1}h_j) = 0,\  \alpha_{k}\not\in \Delta(P)\}.
\end{equation}
We will often simply write $\mf$ when the context is clear.
We consider the more general problem of determining all extremal rays of $\mf_{\Bbb{Q}}=\mf\cap\Gamma_{\Bbb{Q}}(s,K) \subseteq \Gamma_{\Bbb{Q}}(s,K)$. This problem can be refined as follows: Let $L$ be the Levi subgroup of $P$, and $L^{\op{ss}}=[L,L]\subset L$. Note that $L^{\op{ss}}$ is semisimple and simply connected. Let $K(L^{\op{ss}})\subset L^{\op{ss}}$ be the (standard) maximal compact subgroup. We pose ourselves the following more general, related problems:
\begin{enumerate}
\item Describe all extremal rays of $\mf_{\Bbb{Q}}$.
\item Describe $\mf_{\Bbb{Q}}$ in terms of $\Gamma_{\Bbb{Q}}(s,K(L^{\op{ss}}))$.
\end{enumerate}
\subsection{Tensor cones and Eigencones}
For any $\lambda\in \frh^*_{\Bbb{Z}}$ one can associate a line bundle $\ml_{\lambda}$ on $G/B$ (Briefly: $\ml_{\lambda}=G\times_B \Bbb{C}$ as a line bundle on $G/B$ where $\Bbb{C}$ is the $B$ representation given by $\lambda^{-1}$).
Via the Borel-Weil theorem, this sets up a bijection between $\lambda\in \frh^*_{+,\Bbb{Z}}$, the semigroup generated by the dominant fundamental weights and $\Pic^+(G/B)$, the semigroup of line bundles with non-zero global sections: $H^0(G/B,\mathcal{L}_{\lambda})$ is the dual of the irreducible representation $V_{\lambda}$ with highest weight $\lambda$.
\begin{defi}
We have a cone $$\op{Tens}_{s,G,\Bbb{Q}}\subseteq \Pic^+_{\Bbb{Q}}(G/B)^s=(\frh^*_{+,\Bbb{Q}})^s$$ formed by
tuples $(\lambda_1,\dots,\lambda_s)$ such that for some $N>0$, $H^0((G/B)^s,\ml^N)^G \neq 0$ where
$$\ml=\ml_{\lambda_1}\boxtimes \ml_{\lambda_2}\boxtimes\dots\boxtimes \ml_{\lambda_s}\in \Pic(G/B)^s,$$
equivalently, for some $N>0$, $(V_{N\lambda_1}\tensor V_{N\lambda_2}\tensor\dots\tensor V_{N\lambda_s})^G\neq 0$.

\end{defi}
\begin{proposition}\label{bij}
The (Killing form) bijection $\kappa^s:(\frh^*_{\Bbb{Q}})^s\to (\frh_{\Bbb{Q}})^s$ restricts to a bijection between $\op{Tens}_{s,G,\Bbb{Q}}$ and $\Gamma_{\Bbb{Q}}(s,K)$.
\end{proposition}
\begin{proof}
See, for example, \cite{Sjamaar}.
\end{proof}

\subsection{Basic extremal rays}
Fix a face $\mf=\mf(\vec w,P)$ as in Section \ref{regular1} (see Equation \eqref{faceF}), with $P$ an arbitrary standard parabolic subgroup of $G$.
\begin{defi}\cite{BGG}
Let $v,w\in W$ (not necessarily in $W^P$) and $\beta\in R^+$. The notation $v\leto{\beta} w$ stands for
the following two (simultaneous) conditions: $w=s_{\beta}v$ and $\ell(w)=\ell(v)+1$. Note that if $v\leto{\beta} w$,
then $w^{-1}\beta\in R^-$ and $v^{-1}\beta\in R^+$.
\end{defi}
Codimension one Schubert cells $C_v\subseteq X_w$ correspond to $v\leto{\beta} w$ with $v,w\in W^P$ and $\beta$ a positive root (not necessarily simple). As was observed in \cite{BKR}, one should divide the set of such $v$ into two types; this division influences ramification behaviour in intersection theoretic problems:
\begin{defi}\label{RAM}
 Let $w\in W^P$. A codimension one Schubert cell $C_v\subseteq X_w$, $v\in W^P$, is said to be simple if $v\leto{\beta} w$ with $\beta$ a simple root.
\end{defi}

\subsubsection{Divisors in $\left(G/B\right)^s$} In order to construct extremal rays on a face $\mf$ given by (\ref{faceF}), (i.e., line bundles on $\left(G/B\right)^s$, see Proposition \ref{bij}), we will identify a series of $G$ invariant divisors on $\left(G/B\right)^s$, with $G$ acting diagonally on $(G/B)^s$. For a pair $j,v$ such that $v\leto{\beta}w_j$ with $\beta$ simple (i.e, $C_v\subset X_w$ is simple, see Definition \ref{RAM}; in this case $v\in W^P$ automatically), first set
\begin{equation}\label{setU}
u_i = \left\{
\begin{array}{cc}
 w_i, & i\ne j\\
 v, & i=j
\end{array}\right..
\end{equation}
Then define
\begin{equation}\label{defDD}
D(j,v) = \{(\bar g_1,\hdots,\bar g_s)\in (G/B)^s\mid \bigcap_{i=1}^s \bar g_iX_{u_i}\ne \emptyset\}\subseteq (G/B)^s,
\end{equation}
which is given the reduced scheme structure, making it a subvariety of $(G/B)^s$. In Theorem \ref{Tone}, we will show that
$D(j,v)$ is codimension one in $(G/B)^s$. In particular, we can express
\begin{equation}\label{ninth}
\mathcal{O}(D(j,v))=\ml_{\lambda_1}\boxtimes \ml_{\lambda_2}\boxtimes\dots\boxtimes \ml_{\lambda_s}\in \Pic(G/B)^s,
\end{equation}
for some $\lambda_i\in \frh^*_{+,\mathbb{Z}}$.
Since $D(j,v)$ is diagonal $G$ invariant, $H^0((G/B)^s, \mathcal{O}(D(j,v)))^G\neq 0$, and under the bijection of Proposition \ref{bij}, we get
\begin{equation}\label{tenth}
[D(j,v)]=(\kappa(\lambda_1),\dots, \kappa(\lambda_s))\in \Gamma_\Bbb{Q}(s,K)
\end{equation}
The main properties of  $D(j,v)$ are laid out in the following
\begin{theorem}\label{Tone}
\begin{enumerate}
\item[(a)] $D(j,v)$ has codimension one\footnote{This fails without the simpleness condition on $C_v\subset X_{w_j}$, see Remark \ref{ramifyy}.} in $(G/B)^s$;
\item[(b)] $\dim H^0((G/B)^s,\mathcal{O}(mD(j,v)))^G=1$ for all $m\ge 0$;
\item[(c)] $\Bbb{Q}_{\geq 0}[D(j,v)] $ is an extremal ray of $\Gamma_\Bbb{Q}(s,K)$;
\item[(d)] $[D(j,v)] \in \mathcal{F}$.
\end{enumerate}
\end{theorem}

The following  gives formulas for $\lambda_1,\dots,\lambda_s$, and hence for $[D(j,v)]$:
\begin{theorem}\label{timeticks}
Write $\lambda_k=\sum c_{k,\ell} \omega_\ell$. The coefficient $c_{k,\ell}=\lambda_k(\alpha_{\ell}^{\vee})$ is computed as follows.
Fix $k$ and $\ell$ and let $\hat{u}_k=s_{\alpha_\ell}u_k$.
\begin{enumerate}
\item If $\hat{u}_k\in W^P$ and $u_k\leto{\alpha_\ell}\hat{u}_k$, set
$\hat{u}_i=u_i$ for $i\neq k$. Then $c_{k,\ell}$ is the (possibly zero) intersection number
$c$ in
\begin{equation}\label{calc3}
\prod_{i=1}^s [X_{\hat{u}_i}]=c[pt]\in H^*(G/P)
\end{equation}
where the product is in the usual cohomology (and not in the deformed product $\odot_0$, see the example in Section \ref{ex1}).
\item If $\hat{u}_k\not\in W^P$, or $u_k\leto{\alpha_\ell}\hat{u}_k$ is false,
then $c_{k,\ell}=0$.
\end{enumerate}
\end{theorem}
\begin{remark}
We could also have expressed the coefficient $c_{k,\ell}$ as follows: If $\hat{u}_k\in W^P$ and $u_k^{-1}\alpha_\ell\in R^+$
(which is equivalent to $\ell(u_k)<\ell(\hat{u}_k)$), then the coefficient $c_{k,\ell}$ is the number $c$ in \eqref{calc3} and zero otherwise.
\end{remark}
\begin{remark}\label{correspondence}
The pairs $(j,v)$ are in one-one correspondence with simple roots $\beta=\alpha_{\ell}$ and $j$ such that
$w_j^{-1}\alpha_{\ell}\in R^-$. This is because if we set $v=s_{\alpha_{\ell}}w$ then $\ell(w)>\ell(v)$ but $\ell(w)\leq l(v)+1$ since $\alpha_{\ell}$ is a simple root, and hence $\ell(w)=\ell(v)+1$. We have assumed that $w\in W^P$; this implies $v\in W^P$.
\end{remark}

\subsubsection{An example}\label{ex1}
Let $G$ be of type $D_4$, with simple roots $\alpha_1,\alpha_2,\alpha_3,\alpha_4$ (using standard  notation \cite{Bourbaki} here and elsewhere), and corresponding simple reflections $s_i$. Let $P = P_2$ be the standard maximal parabolic for which $\Delta(P) = \Delta\setminus\{\alpha_2\}$. Let $u=s_4s_3s_1s_2$, $v=s_3s_1s_2s_4s_3s_1s_2$, and $w=s_1s_2s_4s_2s_3s_1s_2$; one verifies that $u,v,w\in W^P$, and that
$$
[X_u]\odot_0[X_v]\odot_0[X_w] = [X_e] \in H^*(G/P).
$$
This can be calculated using the multiplication table for $(G/P_2,\odot_0)$ found in \cite{KKM} (in their notation, $[X_u]=\epsilon_{\theta^Pu}^P=b_5^2$, $[X_v]=\epsilon_{\theta^Pv}^P=b_2^3$, $[X_w]=\epsilon_{\theta^Pw}^P=b_2^2$, and $[X_e]=b_9$) and was also verified by computer. Therefore $u,v,w,P$ give rise to a regular facet $\mathcal{F}$.

Observe that, for instance, $s_3v\leto{\alpha_3}v$. According to Theorem \ref{Tone}, $D(2,s_3v)$ is a divisor on $(G/B)^s$, and we now compute the $\lambda_i$ appearing in $\mathcal{O}(D(2,s_3v)) = \ml_{\lambda_1}\boxtimes \ml_{\lambda_2} \boxtimes \ml_{\lambda_3}$ using Theorem \ref{timeticks}. First, $\lambda_1$: testing each of $s_1u,s_2u,s_3u,s_4u$, we see that only $s_2u$ satisfies $\ell(s_2u) = \ell(u)-1$ and $s_2u\in W^P$. As $[X_{s_2u}]\cdot[X_{s_3v}]\cdot[X_w]=[X_e]$, $\lambda_1 = \omega_2$. On may check that this product would equal $0$ if $\odot_0$ were used instead.

For $\lambda_2$: $s_3(s_3v)$ and $s_4(s_3v)$  satisfy the two required conditions (the first is obvious). As $[X_u]\cdot [X_{s_4s_3v}]\cdot[X_w]=0$, $\lambda_2=\omega_3$. Finally $\lambda_3$: only $s_3w$ satisfies the requirements, and $[X_u]\cdot[X_{s_3v}]\cdot[X_{s_3w}]=[X_e]$, so $\lambda_3=\omega_3$. Indeed, $(\omega_2,\omega_3,\omega_3)$ is an extremal ray of the tensor cone for $D_4$, cf. \cite{KKM}, and lies on $\mathcal{F}$. All standard cup product calculations here were done by computer.
\subsection{Other extremal rays}
\begin{defi}
A ray $\Bbb{Q}_{\geq 0}(h_1,\dots,h_s)$ is a type I ray of $\mf_{\Bbb{Q}}$ if there is  a pair $(j,v)$ such that $v\leto{\beta}w_j$ with $\beta$ simple, $v\in W^P$ (i.e, $C_v\subset X_{w_j}$ is simple, see Definition \ref{RAM}) such that $\beta(h_j)>0$.

Rays of  $\mf_{\Bbb{Q}}$ which are not type I are called type II rays of  $\mf_{\Bbb{Q}}$; they span a face $\mf_{2,\Bbb{Q}}$ of  $\mf_{\Bbb{Q}}$: They are defined inside $\mf_{\Bbb{Q}}$ by the system of equalities $\beta(h_j)=0$ whenever $(j,v)$  is a pair such that $v\leto{\beta}w_j$ with $\beta$ simple, $v\in W^P$ (note that $\beta(h_j)\geq 0$ on $\mf$).
\end{defi}
It is an easy consequence of Theorem \ref{timeticks} that the extremal rays $D=D(j,v)$ of Theorem \ref{Tone} are type I (see Corollary \ref{correspondence2} (1)).

Let $\Bbb{Q}_{\geq 0}\delta_1,\dots,\Bbb{Q}_{\geq 0}\delta_q$ be the type I extremal rays of $\mf$ produced by Theorem \ref{Tone}, with $q$ the number of possible $(j,v)$ with $v\leto{\beta}w_j$ and $\beta$ simple. We have a natural cone map
\begin{equation}\label{until}
\prod_{b=1}^q \Bbb{Q}_{\geq 0}\delta_b \times \mathcal{F}_{2,\Bbb{Q}}\to \mathcal{F}_{\Bbb{Q}}
\end{equation}
\begin{theorem}\label{conebij}
The mapping \eqref{until} is an isomorphism of pointed rational cones.
\end{theorem}
Therefore, the general problems enumerated above reduce to the problem of describing $\mathcal{F}_{2,\Bbb{Q}}$:
\begin{enumerate}
\item Describe all extremal rays of $\mf_{2,\Bbb{Q}}$.
\item Describe $\mf_{2,\Bbb{Q}}$ in terms of $\Gamma_{\Bbb{Q}}(s,K(L^{\op{ss}}))$.
\end{enumerate}
We will do this by the process of induction.
\subsection{Induction}
Recall that $L$ is the Levi subgroup of $P$.
\begin{defi}
Given $(h_1,\dots,h_s)\in \frh^s_{L^{\op{ss}},\Bbb{Q}}\subseteq \frh^s_{\Bbb{Q}}$, define
$(y_1,\dots,y_s)=(w_1h_1,\dots,w_s h_s)\in \frh^s_{\Bbb{Q}}$. Finally define $\Ind_{L}^G: \frh^s_{L^{\op{ss}},\Bbb{Q}}\to\frh^s_{\Bbb{Q}}$  by the following formula
$$\Ind_{L}^G(h_1,\dots,h_s)=(y_1,\dots,y_s)-\sum'_{j,v} \frac{2}{(\alpha_{\ell},\alpha_{\ell})}\alpha_{\ell}(y_j)[D(j,v)]\in \frh^s_{\Bbb{Q}}$$
 where the sum is over triples $(j,v)$ such that $v\leto{\alpha_{\ell}}w_j$ with $v\in W^P$ ($[D(j,v)]$ is defined in \eqref{tenth}).
\end{defi}
The induction operation $\Ind_{L}^G$ is constructed geometrically, and in a more general setting. The formulas have
a simpler form in the setting of the equivalent (saturated) tensor invariant problem (as in Proposition \ref{bij}), see Theorem \ref{general_Induction}.
\begin{theorem}\label{IndT}
The restriction of $\Ind_{L}^G$ to $\Gamma_{\Bbb{Q}}(s,K(L^{\op{ss}}))$ defines a surjective mapping of cones
\begin{equation}\label{indsurj}
\Ind_{L}^G:\Gamma_{\Bbb{Q}}(s,K(L^{\op{ss}}))\twoheadrightarrow \mf_{2,\Bbb{Q}}\subseteq \mf_{\Bbb{Q}}\subseteq \Gamma_{\Bbb{Q}}(s,K)
\end{equation}
\end{theorem}
Therefore all extremal rays of $\mf_{2,\Bbb{Q}}$ can be obtained by induction \eqref{indsurj} from extremal rays of
$\Gamma_{\Bbb{Q}}(s,K(L^{\op{ss}}))$. Not  all extremal rays of $\Gamma_{\Bbb{Q}}(s,K(L^{\op{ss}}))$ induct to extremal rays; some may even map to zero. For examples of these phenomena, see Section \ref{app1}.

The surjection  \eqref{indsurj} is of a special type, and the kernel of the associated mapping of vector spaces is controlled by ramification divisors in the associated enumerative problem, see Remark \ref{especial}.
\subsection{Acknowledgements}
We thank Shrawan Kumar for useful discussions, and for his comments and corrections.

\section{Basic extremal rays}
Let $P$ be an arbitrary standard parabolic of $G$. For $w\in W^P$, let $Z_w$ denote the smooth locus of $X_w$. There is a largest group $Q_w\subseteq G$, a standard parabolic, that acts on the closed Schubert variety $X_w$.  Let $Y_w\subseteq X_w$ be the open $Q_w$ orbit in $X_w$. Then $C_w\subseteq Y_w\subseteq Z_w\subseteq X_w$. By \cite{BPo} (also see \cite[Lemma 7.1]{BKR}),
\begin{equation}\label{delta}
\Delta(Q_w)=\Delta_w=\Delta\cap w(R_\frl^+\sqcup R^-).
\end{equation}
Note that if $\alpha\in \Delta(Q_w)$, then $s_{\alpha}\in Q_w$.

By \cite[\S 2.6] {BPo}, (also see \cite[Proposition 7.2]{BKR})
\begin{lemma}\label{Roma}
Let $v\leto{\beta} w$ with $v,w\in W^P$.
The following are equivalent:
\begin{enumerate}
\item $C_v\subseteq Y_w$.
\item $C_v\subseteq X_w$ is simple (i.e., $\beta$ is a simple root).
\end{enumerate}
\end{lemma}
\begin{proof}
This is just a restatement of the formulation in \cite{BKR}. If $C_v\subset Y_w$, \cite[Proposition 7.2]{BKR}
 shows that $\beta$ is a simple root.

From $v\leto{\beta} w$, we know $w^{-1}\beta\in R^-$. If $\beta$ is a simple root then
 $\beta\in \Delta(Q_w)$ (by Equation \ref{delta}), and $C_v\subseteq Y_w$ by \cite[Proposition 7.2]{BKR}.
\end{proof}
\subsection{The geometry}
Define the universal intersection locus
$$
\mathcal{X} = \left\{(\bar g_1,\hdots\bar g_s,\bar{z})\in \left(G/B\right)^s\times G/P \mid \bar{z}\in g_iX_{w_i}~\forall i\right\},
$$
and similarly define subloci $\mathcal{Z}\supseteq\mathcal{Y}\supseteq\mathcal{C}$ using the $Z_{w_i}, Y_{w_i}, C_{w_i}$, respectively, in place of the $X_{w_i}$. (For the scheme structures, see \cite{BKR}.) As shown in \cite{BKR}, $\mathcal{X}$ is irreducible and the projection map $\pi:\mathcal{X} \to \left(G/B\right)^s$ is birational.
Moreover, because the $X_w$ are all normal, each $X_{w_i}\setminus Z_{w_i}$ has codimension $\ge 2$ in $X_{w_i}$; accordingly, $\mathcal{X}\setminus\mathcal{Z}$ has codimension $\ge 2$ in $\mathcal{X}$.

%So we may restrict our attention to divisors in the smooth variety $\mathcal{Z}$.
Let $\mathcal{R}\subseteq \mathcal{Z}$ be the ramification divisor of the map $\pi$. It follows from the birationality of $\pi$ that $\overline{\pi(\mathcal{R})}\subseteq (G/B)^s$ is codimension $\geq 2$ in $(G/B)^s$.
Furthermore, it was demonstrated in \cite{BKR} that $\mathcal{Z}\setminus\mathcal{Y}\subset \mathcal{R}\cup A$ for some $A\subset \mathcal{Z}$ of codimension at least $2$.

\begin{defi}\label{gradechange}
With notation as in \eqref{setU},  let
$$\widetilde{D}(j,v) = \{(\bar g_1,\hdots,\bar g_s,\bar{z})\in (G/B)^s\times G/P\mid \bar{z}\in \bigcap_{i=1}^s \bar g_iX_{u_i}\}\subseteq \mathcal{X}\subseteq (G/B)^s \times G/P.
$$
Note that $\widetilde{D}(j,v)$ has a natural scheme structure. Clearly, as sets, ${D}(j,v)=\pi(\widetilde{D}(j,v))$. We will see in Corollary \ref{nom} that this is an equality of divisors. It is easy to see that $\widetilde{D}(j,v)$ is irreducible by the same argument as for the irreducibility of $\mathcal{X}$ (see e.g., \cite{BKR}).
\end{defi}

\subsection{Proof of Part (a) of Theorem \ref{Tone}}\label{refer}
Assume
\begin{itemize}
\item $j=1$ and $w_1=s_{\beta}v$ with $\beta$ a simple root. Set $\widetilde{D}=\widetilde{D}(1,v)$.
\end{itemize}
First choose a point $(\bar{g}_1,\dots,\bar{g}_s,\bar{z})\in \mathcal{C}-\mathcal{R}$. This implies that
(by birationality of $\pi$ and Zariski's main theorem)
\begin{equation}\label{ZM}
\bigcap_{i=1}^s g_iX_{w_i}=\{\bar{z}\}.
\end{equation}

Assume first $\bar{z}= w_1P$, by translations. We get $w_1= g_1 bw_1p$ for some $b\in B$ and $p\in P$. Assume $b=e$ without changing $\bar{g}_1$. Therefore $g_1\in w_1Pw_1^{-1}$, translate $(\bar{g}_1,\dots,\bar{g}_s,\bar{z})$ by $g_1^{-1}$ which fixes $w_1P$. Therefore we may assume $g_1=e$ and $\bar{z}=w_1P$.

Let $g_1'=s_{\beta}$. Then $g_1'X_{w_1}= g_1 X_{w_1}$ since $s_{\beta}\in Q_{w_1}$.
Now $\bar{z}=w_1P= s_{\beta} vP$. Therefore
$\bar{z}\in g'_1 C_{v}$ and $(\bar{g}'_1,\bar{g}_2,\dots,\bar{g}_s,\bar{z})\in \widetilde{D}\cap \mathcal{Y}$. Now,
$$g'_1 X_{w_1}\cap \bigcap_{i=2}^s g_iX_{w_i}=\bigcap_{i=1}^s g_iX_{w_i}=\{\bar{z}\}$$
which shows that $(\bar{g}'_1,\bar{g}_2,\dots,\bar{g}_s,\bar{z})\in \widetilde{D}\cap \mathcal{Y}-\mathcal{R}$. The divisor $\widetilde{D}$ is therefore not contained in $\mathcal{R}$. This finishes the proof of (a). The above argument has the following corollary:
\begin{corollary}\label{nom}
As Weil divisors, ${D}(j,v)=\pi_*(\widetilde{D}(j,v))$. Here $\pi_*$ is the push forward operation on cycles under proper morphisms \cite[Section 1.4]{Fint}.
\end{corollary}
\begin{proof}
Since $\pi$ is an isomorphism on $\mathcal{Z}-\mathcal{R}$, we only need to observe that $\widetilde{D}(j,v)$ is irreducible and generically smooth. Both are standard  (and follow by studying the fiber bundle $\widetilde{D}(j,v)\to G/P$ as in \cite{BKR}).
\end{proof}

\subsection{A theorem on  functions on the universal intersection}
\begin{theorem}\label{equity}
$$H^0(\mathcal{C}-\mathcal{R},\mathcal{O})^G=\Bbb{C}.$$
\end{theorem}

This will be proved in Section \ref{fin1}.

\subsection{Proof of Parts (b) and (c) of Theorem \ref{Tone}}
Part (c) follows from part (b) as in \cite[Lemma 2.1]{BHermit} (briefly): Suppose some multiple $\mathcal{O}(mD)$ is a tensor product of two line bundles $\ml_1$ and $\ml_2$ with non-zero invariants $s_1$ and $s_2$. Clearly the zero sets of $s_1$ and $s_2$ should be supported on $D$, since $\mathcal{O}(mD)$ has only one non-zero invariant section (by (b)) up to scaling, which has this property. Since $D$ is reduced, and $(G/B)^s$ is smooth, $\ml_1$ and $\ml_2$ are both multiples of $\mathcal{O}(D)$, as desired.

We deduce part (b) from Theorem \ref{equity}. For this we use
\begin{lemma}\label{malta}
$\pi(\mathcal{C}-\mathcal{R})\subseteq (G/B)^s-{D}(j,v)$
\end{lemma}
\begin{proof}Set $$
u_i = \left\{
\begin{array}{cc}
 w_i, & i\ne j\\
 v, & i=j
\end{array}\right..
$$
If $p=(\bar{g}_1,\dots,\bar{g}_s)\in D(j,v)$ is the image of a point $q=(\bar{g}_1,\dots,\bar{g}_s,\bar{z})\in \mathcal{C}-\mathcal{R}$, then $\pi^{-1}(p)$ is isolated at $q$ (here $\pi:\mathcal{X}\to (G/B)^s$), and contains a different point $(\bar{g}_1,\dots,\bar{g}_s,z')$ with $z'\in \bigcap_{i=1}^s \bar g_iX_{u_i}$. This contradicts Zariski's main theorem.
\end{proof}

We now prove part (b) of Theorem \ref{Tone}: $f\in H^0((G/B)^s,\mathcal{O}(mD))^G$ gives a $G$ invariant function, $f$, on  $(G/B)^s-D$, and hence one on $\mathcal{C}-\mathcal{R}$ (using Lemma \ref{malta}): namely, $f\circ \pi$. By Theorem \ref{equity}, $f\circ\pi$ is constant on $\mathcal{C}-\mathcal{R}$; thus $f$ is constant on $\pi(\mathcal{C}-\mathcal{R})$, a dense open subset of $(G/B)^s-D$. Therefore $f$ must be constant on all of $(G/B)^s-D$, which proves the claim.

Theorem \ref{Tone}, part (d) is proved in Section \ref{fin2}.

\section{Some parameter spaces}\label{sps}
\subsection{Principal $G$-spaces}
We first point out why a stack theoretic approach is convenient. Suppose $G=\op{SL}(n)$ and $P$ a maximal parabolic with $G/P=\op{Gr}(r,n)$, the Grassmannian of $r$-dimensional subspaces of $\Bbb{C}^n$. Then $\mathcal{C}$ parameterizes the set of $r$-dimensional subspaces $V$ of $\Bbb{C}^n$ and $s$ full flags of subspaces such that $V$ is in prescribed Schubert cells with respect to these flags. Now one can consider induced flags on such a $V$, and on $Q=\Bbb{C}^n/V$. But $V$ and $\Bbb{C}^n/V$ are not $\Bbb{C}^r$ and $\Bbb{C}^{n-r}$ canonically; therefore we do not get a map to a product of flag varieties. Stacks provide a convenient setting to still make this work to pull back objects defined invariantly. We may just study pairs $V$ and $Q$ with $s$ full flags on each. This is a stack and we do get a map from  $\mathcal{C}$ to this stack. But we will have to work with objects without trivializations, hence with principal $G$ spaces (principal bundles over a point).

\subsection{Principal spaces and relative positions}
A principal $G$ space is a variety $E$ with a right $G$ action that is principal homogenous for the action of $G$ (i.e., for any $x\in E$ the map $G\to E$ given by $g\mapsto xg$ is an isomorphism). If $\phi:G\to G'$ is a map of affine algebraic groups and $E$ a principal $G$-space then $E\times_G {G'}=E\times_{\phi} G'$ is a principal $G'$ space.

Suppose $\bar{g}\in E/B$ and $\bar{z}\in E/P$. We define the relative position $[\bar{g},\bar{z}]\in W^P$ as follows. It is the element $w\in W^P$ such that
\begin{equation}\label{indicates1}
z=gbwp^{-1}, \text{ for some } b\in B, p\in P.
\end{equation}
Here $g,z\in E$ are coset representatives of $\bar{g}\in E/B$ and $\bar{z}\in E/P$.
It is easy to see that $w$ is independent of choices. If we choose a trivialization $e\in E$, we get corresponding elements $\bar{g}\in G/B$ and
$\bar{z}\in G/P$. Equation \eqref{indicates1}  indicates that $\bar{z}\in G/P$
is in the Schubert cell $gBwP/P$.

\subsection{Good representatives}
Throughout this paper, we choose a set theoretic lifting $W\to N(T)$ of $W\leto{\sim} N(T)/T$.
Suppose $\bar{g}\in E/B$ and $\bar{z}\in E/P$ with $w=[\bar{g},\bar{z}]\in W^P$. Consider $p$ as in \eqref{indicates1}. Write $z=gbwp^{-1}$ as $zp=gbw$, and change $z$ to $zp$ and $g$ to $gb$. The equation simplifies to $z=gw$. Therefore we may choose a (``good'') representative $(g,z)$ of $(\bar{g},\bar{z})$ so that $z=gw$.
The choice of ``good representative" is unique up to the action of   $(w^{-1}Bw\cap P)$: If $(zp,gb)$ is  another choice of a good representative then $z=gw$ and $zp=gbw$ and hence $gwp=gbw$ and hence $p=w^{-1}bw\in (w^{-1}Bw\cap P)$.
\begin{defi}\label{three}
Let $E$ be a principal $G$-space. An element $\bar{z}\in E/P$ defines a principal $P$-space $E_P(\bar{z})$ (the coset in $E/P$), and hence a principal $L$ space $E_L(\bar{z})=E_P(\bar{z})\times_P L$, using the  quotient map $P\to L=P/U$.
\end{defi}

\begin{lemma}\label{else}
Under the map $P\to L$, the subgroup $w^{-1}Bw\cap P$ maps to $B_L\subseteq L$, in fact onto it.
\end{lemma}
\begin{proof}
First we show that $w^{-1}Bw\cap P$ is connected: $B\cap wPw^{-1}= T\cdot (U\cap wPw^{-1})$ since both
$B$ and $wPw^{-1}$ contain $T$. Now $T$ acts on both $U$ and $wPw^{-1}$ by conjugation.
By \cite[Section 14.4, Proposition (2)]{Borel} applied to $U\cap wPw^{-1}\subset U$, we see that $U\cap wPw^{-1}$, and hence $B\cap wPw^{-1}$ is connected.

Clearly $wB_Lw^{-1}\subseteq B$. Therefore $B_L\subseteq w^{-1}Bw$ and $B_L\subseteq P$, therefore
$B_L\subseteq w^{-1}Bw\cap P$. Since $w^{-1}Bw\cap P$ is connected we may prove the mapping property at the
level of Lie algebras, which is easy because $w\alpha<0$ for any negative root $\alpha$ in $R_L$.
\end{proof}

\subsection{Universal Schubert intersection stacks}\label{parallel}

We introduce the following stacks:

\begin{enumerate}
\item The stack $\op{Fl}_G$ parameterizes principal $G$-spaces $E$ together with elements $\bar{g}_i\in E/B, i=1,\dots,s$ (i.e., to consider families of such objects over a scheme $X$, we consider principal $G$-bundles $E$ on $X$ locally trivial  in the fppf topology, and sections $\bar{g}_i$ of $E/B$ over $X$). It is easy to see that $\op{Fl}_G=(G/B)^s/G$ (the stack quotient here, and below). Let $\op{Fl}_L=(L/B_L)^s/L$, which parameterizes principal $L$-spaces $F$ together with elements $\bar{q}_i\in F/B_L, i=1,\dots,s$.

\item The moduli stack $\hat{\mathcal{C}}$ which parameterizes principal $G$-spaces $E$ together with elements $\bar{g}_i\in E/B$ and a single element $\bar{z}\in E/P$ so that $[\bar{g}_i,\bar{z}]=w_i$ for all $i$. It is easy to see that $\hat{\mathcal{C}}=\mathcal{C}/G$. There is a natural map $\pi:\hat{\mathcal{C}}\to \op{Fl}_G.$
\end{enumerate}

\begin{lemma}\label{four}
Let $E$ be a principal $G$-space, and $\bar{g}\in E/B$ and $\bar{z}\in E/P$ with $w=[\bar{g},\bar{z}]\in W^P$. Consider $p$ as in \eqref{indicates1}.
The element $zp\in E_P(\bar{z})/(w^{-1}Bw\cap P)$ is well defined. As a result, the corresponding element in $E_L(\bar{z})/B_L$ is well defined (see Lemma \ref{else}).
\end{lemma}

\begin{lemma}\label{zusammen}
The stack $\hat{\mc}$ parameterizes principal $P$ spaces $E'$ together with $\bar{z}_i\in E'/(w_i^{-1}Bw_i\cap P), i=1,\dots,s.$
\end{lemma}
\begin{proof}
This follows from Lemma \ref{four}: From $(E',\bar{z}_1,\dots,\bar{z}_s)$, we let $E=E'\times_{P} G$ and
$g_i=z_i\times w_i^{-1}$. The point $\bar{z}\in E/P$ is the tautological point. Lemma \ref{four} gives the reverse correspondence, where $E':=zP$ and $\bar{z_i}:=\overline{zp_i}$.
\end{proof}

\subsection{A diagram of spaces}\label{gram}
Lemmas \ref{zusammen} and \ref{else} and the map $P\to L$ result in a map   $\tau:\hat{\mc}\to \Fl_L$. The inclusion
$L\to P$ gives rise to $i:\op{Fl}_L\to \hat{\mathcal{C}}$, so that $\tau\circ i$ is the identity on $\Fl_L$. The map $\tilde i$ is $\pi\circ i$.
\begin{equation}\label{diagrammo}
\xymatrix{
\mathcal{C}\ar[r]\ar[d]^{\pi} & \hat{\mathcal{C}}\ar[d]^{\pi}\ar[dr]_{\tau}  \\
(G/B)^s\ar[r] & \op{Fl}_G & \op{Fl}_L\ar@/_/[ul]_i\ar[l]^{\tilde{i}}}
\end{equation}

Note that $\tilde{i}$ sends a tuple $(F,\bar{l}_1,\dots,\bar{l}_s)$ to $(E,\bar{g}_1,\dots,\bar{g}_s)$ where $E=F\times_L G$, and $g_i=l_i \times w_i^{-1}$

\subsection{Levification}\label{secLev}
\begin{defi}
Let $Z^0(L)\subseteq Z(L)$ be the connected component of the identity of the center $Z(L)$ of $L$.
\end{defi}

The space $\hat{\mc}$ retracts to $\Fl_L$ by a process called Levification (cf. \cite[Section 3.8]{BKq}, also \cite[Prop 3.5]{Ram}). Consider an element $x_L=\sum'_{k} N_kx_k$ where the sum is over $k$ such that $\alpha_k\not\in \Delta(P)$, with $N_k$ such that $N_kx_k$ is in the co-root lattice. Then $t^{x_L}=\exp((\ln t) x_L)$ topologically generates $Z^0(L)$ as $t$ varies  in $\Bbb{C}^*$.

For $t\in\Bbb{C}^*$ we have an automorphism $\phi_t:P\to P$ given by (cf. \cite[Section 3.8]{BKq} and \cite[Lemma 3.1.12]{Ram}) $\phi_t(p)=t^{x_L}pt^{-x_L}$, with  $\phi_1$ the identity on $P$. This extends to a group homomorphism $\phi_0:P\to L$ (which coincides with the standard projection of $P$ to $L$) giving rise to a morphism $\hat{\phi}:P\times \Bbb{A}^1\to P$. Clearly, $\phi_t:L\to L$ is the identity on $L$ for all $t$.

\begin{defi}\label{levity}
Let $(E',\bar{z}_1,\dots,\bar{z}_s)$ be a point of $\hat{\mc}$ where $E'$ is a principal $P$-space and $\bar{z}_i\in
E'/(w_i^{-1}Bw_i)\cap P$. Define the Levification family $E_t'=E'\times_{\phi_t} P$ for $t\in \Bbb{A}^1$, and $\bar{z}_i(t)=\bar{z}_i\times_ {\phi_t} e$. Clearly, at $t=0$, $(E_t',\bar{z}_1(t),\dots,\bar{z}_s(t))$ is in the image of $i:\Fl_L\to \hat{\mc}$, and equals $i\circ\tau((E',\bar{z}_1,\dots,\bar{z}_s))$.
\end{defi}
Consider a point $(\bar{g}_1,\dots,\bar{g}_s,\bar{z})\in \mathcal{C}$, which gives a point of $\hat{\mc}$. Write equations
$g_ib_iwp_i^{-1}=z$, or $g_ib_iw= zp_i$. We get the principal $P$ space $E'=zP$ (independent of the lift of $\bar{z}$) and well defined points $z_i=zp_i\in E'/(w_i^{-1}Bw_i\cap P)$. All the spaces $E_t'=E'\times_{\phi_t} P$ are trivialized by $z$. Under this trivialization $(E_t',\bar{z}_1(t),\dots,\bar{z}_s(t))$ is the point
$$(P,\phi_t(\bar{z}_1),\phi_t(\bar{z}_2),\dots, \phi_t(\bar{z}_s))= (P, t^{x_L}p_1t^{-x_L}, \dots, t^{x_L}p_st^{-x_L})$$
The corresponding $E$ spaces are also trivial, and hence we obtain a lifting of this part in $\mc$: the
points  $(t^{x_L}p_1 w_1^{-1}, \dots, t^{x_L}p_s w_{s}^{-1})$:  at $t=0$ we get
$(p_1 w_1^{-1}, \dots, p_s w_{s}^{-1})=z^{-1}(g_1b_1,\dots,g_sb_s)$

\subsection{Comparison of line bundles and sections}

\begin{defi}\label{invariant}
Let $\mathcal{M}$ be a line bundle on $\Fl_L$.  $Z(L)$ acts on fibers  of $\mathcal{M}$ and gives rise to a (multiplicative) character $\gamma_{\mathcal{M}}:Z^0(L)\to \Bbb{C}^*$ (note that the group of characters of $Z^0(L)$ is discrete).  More generally, this map can be defined if $\mathcal{M}$ is defined over an open substack  of $\Fl_L$ since $\gamma_{\mathcal{M}}$ is constant over connected families.
\end{defi}
\begin{proposition}\label{comparison}
Let $U$ be a non-empty open substack of $\Fl_L$,  $\ml$ be a line bundle on $\tau^{-1}(U)$ and $\mathcal{M}=i^*\mathcal{L}$, a line bundle on $U$, where $i:\Fl_L\to \hat{\mc}$ and $\tau:\hat{\mc}\to \Fl_L$.
Then
\begin{enumerate}
\item $\ml=\tau^*\mathcal{M}$. Therefore $\tau^*$ and $i^*$ set up isomorphisms $\Pic(U)\leto{\sim}\Pic(\tau^{-1}(U))$.
\item If $\gamma_{\mathcal{M}}$ is trivial then
$H^0(\tau^{-1}(U),\ml)\to H^0(U,\mathcal{M})$ is an isomorphism.
\end{enumerate}
\end{proposition}
\begin{proof}
The second part is essentially  \cite[Theorem 15 and Remark 31(a)]{BK}, and the main point is that if $E_t$ is a Levification family then a section of $\ml$ (under the assumption of (2)) at $E_1$ can be propagated in a unique way to all $E_t$, $t\neq 0$ (since $E_1$ is isomorphic to $E_t$ for $t\neq 0$), and there are no poles or zeroes of this extended section at $t=0$.  For the surjectivity we can extend any section of $\ml$ at $E_0$ to all of $E_t$ since the corresponding $\Bbb{C}^*$-equivariant line bundle on $\Bbb{A}^1$ is trivial.

For the first part consider $\ml'=\ml\tensor \tau^*{\mathcal{M}}^{-1}$. Note that $\mathcal{M}'= i^* \mathcal{L}'$ is trivial, and
$\gamma_{\mathcal{M}'}$ is trivial. We can apply (2) to $(\ml',\mathcal{M}')$. The nowhere vanishing global section of $H^0(\Fl_L,\mathcal{M}')$ gives a global section of $H^0(\hat{\mc},\ml)$. It can be seen that this is nowhere vanishing as  well (see \cite[Lemma 3.17]{BKq}: Consider the corresponding Levification family (Definition \ref{levity}), if a global section vanishes at $E_t$ then it will also vanish at $E_0$.
\end{proof}
\subsection{Notation for Picard groups of stacks}
Let $\mathcal{X}$ be a stack (for us $\mathcal{X}$ will be a quotient stack)
\begin{itemize}
\item $\Pic(\mathcal{X})$ will denote the Picard group of line bundles on $\mathcal{X}$.
\item $\Pic_{\Bbb{Q}}(\mathcal{X})=\Pic(\mathcal{X})\tensor\Bbb{Q}$.
\item $\Pic^+(\mathcal{X})\subseteq \Pic(\mathcal{X})$ is the monoid of line bundles with non-zero global sections.
\item $\Pic^+_{\Bbb{Q}}(\mathcal{X})\subseteq \Pic_{\Bbb{Q}}(\mathcal{X})$ is the subset of elements such that some multiple has a non-zero global section.
\end{itemize}
\subsection{Ramification divisors}\label{ramify}
\begin{defi} (cf. \cite[Section 4]{BKR})
Consider a linear map $p : V \to W$ between vector spaces of the same dimension $m$. Let
$$\det(p) := (\bigwedge^{m}V)^* \tensor (\bigwedge^{m} W) = Hom(\bigwedge^{m} V,\bigwedge^{m}W).$$
Denote by $\theta(p)$ (``the theta section") the canonical element of $\det(p)$ induced by the top exterior power of p.
\end{defi}
At a point $a=(\bar{g}_1,\bar{g}_2,\dots,\bar{g}_s,\bar{z})$ of $\mathcal{C}$, we may consider two maps
$$T\mathcal{C}_a\to T(G/B)^s_{(\bar{g}_1,\bar{g}_2,\dots,\bar{g}_s)}$$
and
$$T(G/P)_{\bar{z}}\to \bigoplus \frac{T(G/P)_{\bar{z}}}{T(g_iC_{w_i})_{\bar{z}}}$$
The theta sections and determinant lines of the two maps above are isomorphic \cite[Lemmas 4.1 and 4.2]{BKR}, and describe the ramification divisor and the associated line bundle of the birational $\pi:\mathcal{C}\to (G/B)^s$.
They give rise to the line bundle $\mathcal{O}(\mathcal{R})$ and the divisor $\mathcal{R}$ on $\mathcal{C}$.

The line bundle $\mathcal{O}(\mathcal{R})$ is the pull back of a natural line bundle on $\hat{\mc}$: Consider a
$P$ bundle $E'$ together with $\bar{z}_i\in E'/(w_i^{-1}Bw_i\cap P)$. This gives a point of $\hat{\mc}$ as in Lemma \ref{zusammen}. Consider the map of vector spaces
\begin{equation}\label{eseq}
E'\times_P T(G/P)_{\dot e}\to \bigoplus \frac{E'\times_P T(G/P)_{\dot{e}}}{z_i\times T(w_i^{-1}C_{w_i})_{\dot{e}}}
\end{equation}
The determinant line and theta section for this map give a line bundle and a section on  $\hat{\mc}$, denoted
by $\mathcal{O}(\hat{\mathcal{R}})$ and  $\hat{\mathcal{R}}$. Recall that if $E$ is a principal $P$ space and $P\to\GL(V)$ a homomorphism ($V$ a vector space), then $E\times_P V$ is associated vector space (a quotient of $E\times P$ by an appropriate action of $P$).

We claim that $\mathcal{O}(\hat{\mathcal{R}})$ and  $\hat{\mathcal{R}}$ pull back to $\mathcal{O}(\mathcal{R})$ and the divisor $\mathcal{R}$ on $\mathcal{C}$. To see this consider a $(\bar{g}_1,\dots,\bar{g}_s,\bar{z})\in \mathcal{C}$. Let $E'=zP$ and find good pairs $(g_i,z_i)$  lifting $(\bar{g_i},\bar{z})$. There is a natural map
$xP\times_P  T(G/P)_{\dot{e}}\to T(G/P)_{\bar{z}}$ and similarly maps
$$ \frac{E'\times_P T(G/P)_{\dot{e}}}{z_i\times T(w_i^{-1}C_{w_i})_{\dot{e}}}\to \frac{T(G/P)_{\bar{z}}}{T(g_iC_{w_i})_{\bar{z}}}.$$
The claim now follows from the functoriality of the determinant line and its theta section  \cite[Lemmas 4.1 and 4.2]{BKR}.

Pulling back $\mathcal{O}(\hat{\mathcal{R}})$ and  $\hat{\mathcal{R}}$ via $i$ we get a divisor $\mathcal{R}_L$  and line bundle $\mathcal{O}(\mathcal{R}_L)$ on $\Fl_L$.

\begin{remark}\label{deform}
 Under the assumption (\ref{kishi}), $for \mathcal{M}=\mathcal{O}(\mathcal{R}_L)$, $\gamma_{\mathcal{M}}$ is trivial (see Definition \ref{invariant}):  This follows from the definition of the deformed product given in \cite{BK}. Start with $(\bar{l}_1,\dots,\bar{l}_s)\in (L/B_L)^s$; one gets a point of $\Fl_L$. Then the fiber of $\mathcal{O}(\mathcal{R}_L)$ is identified (using \eqref{eseq}, $E'$ trivial) with the determinant line of
\begin{equation}
T(G/P)_{\dot e}\to \bigoplus \frac{T(G/P)_{\dot{e}}}{T(l_i w_i^{-1}C_{w_i})_{\dot{e}}}
\end{equation}
The center of $L$ acts on the trivial $E'$ by automorphisms, and therefore we have an action of $Z^0(L)$ on the determinant line of this morphism. The triviality of this action is therefore implied by the assumption \eqref{kishi}.
\end{remark}

Since we have assumed non-zeroness in the deformed product, by Proposition  \ref{comparison},
\begin{corollary}\label{carol}
\begin{enumerate}
\item $\mathcal{O}(\hat{\mathcal{R}})=\tau^* \mathcal{O}(\mathcal{R}_L)$
 \item $\hat{\mathcal{R}}=\tau^{-1} \mathcal{R}_L$ so that  $i:\Fl_L-\mathcal{R}_L\to \hat{\mathcal{C}}-\hat{\mathcal{R}}$.
\end{enumerate}
\end{corollary}
Finally we recall the generalization of Fulton's conjecture proved in \cite{BKR} in two equivalent forms:
\begin{proposition}\label{fultonc}
\begin{enumerate}
\item $\dim H^0(\Fl_L,\mathcal{O}(m\mathcal{R}_L))=1$ for all $m\geq 0$.
\item $H^0(\Fl_L-\mathcal{R}_L,\mathcal{O})=\Bbb{C}$.
\end{enumerate}
\end{proposition}
\subsection{Proof of Theorem \ref{equity}}\label{fin1}
Since the pull-back of $\hat{\mathcal{R}}$ to $\mathcal{C}$ is $\mathcal{R}$,
$$
H^0(\mathcal{C}-\mathcal{R},\mathcal{O})^G=H^0(\hat{\mathcal{C}}-\hat{\mathcal{R}},\mathcal{O}),
$$
and the theorem reduces to showing $H^0(\hat{\mathcal{C}}-\hat{\mathcal{R}},\mathcal{O}) = \mathbb{C}$.

Then note that from Corollary \ref{carol},
$H^0(\hat{\mc}-\hat{\mathcal{R}},\mathcal{O})= H^0(\hat{\mc}-\tau^{-1}\mathcal{R}_L,\mathcal{O})$.
Applying Proposition \ref{comparison} (2) on $\ml=\mathcal{O}$ ($\gamma_{\mathcal{O}}$ is trivial) and $U=
\Fl_L-\mathcal{R}_L$, we see that $H^0(\hat{\mc}-\tau^{-1}\mathcal{R}_L,\mathcal{O})=H^0(\Fl_L-\mathcal{R}_L,\mathcal{O})$. The last space is
$\Bbb{C}$ by Proposition \ref{fultonc}.
\subsection{Proof of Theorem \ref{Tone}(d)}\label{fin2}
For (d), we follow arguments of Ressayre \cite{R1}: First it is enough to show (see Lemma \ref{ress} below) that the pull back of $\mathcal{O}(D)$ along $\Fl_L\to \Fl_G$ has a non-zero global section on $\Fl_L$. But this is clear since the pull back section $1$ does not vanish along $\Fl_L-\mathcal{R}_L$  because $\Fl_L-\mathcal{R}_L$ does not meet $D$ by Proposition \ref{malta}.

\begin{lemma}\label{ress}
Suppose $\ml=\ml_{\lambda_1}\boxtimes \dots \boxtimes \ml_{\lambda_s}\in \Pic(G/B)^s=\Pic(\Fl_G)$. Let $\mathcal{M}$ be the pull back to $\Fl_L$ of $\ml$ (via $\tilde{i}$). Then the following are equivalent:
\begin{enumerate}
\item $x$ satisfies the linear equalities defining the face $\mf$, i.e., setting $x=(\kappa(\lambda_1)\dots,\kappa(\lambda_s))= (h_1,\dots,h_s)\in \frh^s$, we have
$$\sum_{j=1}^s \omega_{{k}}(w_j ^{-1}h_j) = 0,\ \ \forall \alpha_{k}\not\in \Delta(P).$$
\item $\gamma_{\mathcal{M}}:Z^0(L)\to \Bbb{C}^*$ is trivial.
\end{enumerate}
Furthermore, if $H^0(\Fl_L,\mathcal{M})\neq 0$ then the equivalent conditions above hold.
\end{lemma}
\begin{proof}
Since $Z^0(L)$ is connected (2) is equivalent to $\sum\lambda_j(w_jx_k)=0$ for all $\alpha_k\in \Delta-\Delta(P)$ which is equivalent to (1).

If $s\neq 0\in H^0(\Fl_L,\mathcal{M})$, then the action of the center of $L$ on $s$ is trivial on $\mathcal{M}$ at any point where $s$ is not zero, and hence $\gamma_{\mathcal{M}}:Z^0(L)\to \Bbb{C}^*$ is trivial.
\end{proof}

\section{Divisor classes}
\subsection{Schubert cells are affine spaces}
Let $v\in W^P$, let $U_v= \{u\in U| v^{-1}u v\in  U^-\}=(vU^{-}v^{-1})\cap U$. Then the map
$U_v\to C_v$ which takes $u\to uv\dot{e}\in G/P$ is an isomorphism (see \cite[Proposition 5.1]{BGG} and the references therein).

Let $\Phi_v$ be the set of positive roots $\alpha$ such $ v^{-1}\alpha\in R^-$. Then the product mapping (in any order, different orders give different mappings)
$$\prod_{\alpha\in \Phi_v} U_{\alpha} \leto{\sim} U_v \leto{\sim} C_v$$
is an isomorphism of varieties \cite[Section 14.4]{Borel}. Here $U_{\alpha}\leto{\sim}\Bbb{G}_a\leto{\sim}\Bbb{A}^1$ is the subgroup corresponding to the positive root $\alpha$.

Assume in the lemmas below that $v,w\in W^P$ and $v\leto{\beta} w$, with $\beta=\alpha_{\ell}$ a simple root.
\begin{lemma}\label{super}
\begin{enumerate}
\item $s_{\beta}\Phi_v\subseteq \Phi_w$.
\item  $\Phi_w-s_{\beta}(\Phi_v)=\{\beta\}$.
\end{enumerate}
\end{lemma}
\begin{proof}
For (1): If $\alpha\in \Phi_v$, then we need to show that
$\gamma=s_{\beta}\alpha\in \Phi_w$. We know $\alpha\neq \beta$ since
$v^{-1}\beta\in R^+$. Therefore $\gamma=s_{\beta}\alpha\in R^+$ (see \cite[Lemma D.25]{FH}), and $w^{-1}\gamma = v^{-1}s_{\beta}s_{\beta}\alpha=v^{-1}\alpha\in R^-$ since $\alpha\in \Phi_v$.

For (2): From $v\leto{\beta} w$, we find that $w^{-1}\beta\in R^-$, and hence $\beta\in \Phi_w$. We claim
$\beta\not\in s_{\beta}\Phi_v$.  If $\beta=s_{\beta}\alpha$, then $\alpha=s_{\beta}\beta=-\beta\in R^-$, and therefore we are done using $|\Phi_v|=\ell(v)$ (similarly for $w$) and $\ell(w)=\ell(v)+1$.
\end{proof}

\begin{lemma}\label{ash}
\begin{enumerate}
\item  $s_{\beta} C_v\subseteq C_w$.
\item The map $\Bbb{A}^1\times s_{\beta}C_v\to G/P$ which sends $(t,s_{\beta}x)$ to $\exp(tE_{\beta})s_{\beta}x$ sets up an isomorphism $\Bbb{A}^1\times s_{\beta} C_v\to C_w$.
\end{enumerate}
\end{lemma}
\begin{proof}
The first statement follows from the first part of Lemma \ref{super}.  The second part follows from \cite[Section 14.4]{Borel}.
\end{proof}
Recall that $Q_w$ is the largest subgroup of $G$ that preserves $X_w$. Since $w^{-1}\beta\in R^-$ and $\beta$ is simple, $s_{\beta}$ in $Q_w$ and $s_\beta C_w\subseteq X_w$. The inclusion $s_{\beta} C_v\subseteq C_w$ therefore yields a factorization of the canonical inclusion: $C_v\subseteq s_{\beta}C_w\subseteq X_w$.

\subsection{Universal Schubert varieties and their cycle classes}
Let $u\in W^P$ and consider the universal Schubert variety
$$S_u=\{(\bar{g},\bar{z})\mid z\in \bar{g}X_u\}\subseteq G/B\times G/P.$$
Recall that $X_u\subseteq G/P$ is the closure of the Schubert cell $C_u$.
Let $m=\dim G/P- \ell(u)$, the codimension of $X_u$ in $G/P$.

We want to determine the first two terms ($j=0,1$ below)
of the cycle class $[S_u]\in A^*(G/B\times G/P)$ of $S_u$ in the
decomposition
$$A^m(G/B\times G/P)=\bigoplus_{j=0}^m A^{j}(G/B)\tensor A^{m-j}(G/P)$$

We may intersect with $[\dot{e}]\times g[X_w]$, with $g$ general, and $w\in W^P$ arbitrary such that
$\ell(w)=\dim G/P-\ell(u)$ and see that
the $j=0$ term is $1 \tensor [X_u]$.

Write  the $j=1$ term as
$\sum_{\ell} \ml_{\omega_{\ell}}\tensor \beta_{\ell}$.

\begin{proposition}\label{culture}
Let $\hat{u}_{\ell}=s_{\alpha_{\ell}}u$.
\begin{enumerate}
\item If $\hat{u}_{\ell}\not\in W^P$ or $u\leto{\alpha_{\ell}}\hat{u}_{\ell}$ is false then $\beta_{\ell}=0$.
\item If $\hat{u}_{\ell}\in W^P$ and $u\leto{\alpha_{\ell}}\hat{u}_{\ell}$ then $\beta_{\ell}=[X_{\hat{u}}]$.
\end{enumerate}
\end{proposition}
\subsection{Proof of Proposition \ref{culture}}

For every simple root $\beta=\alpha_{\ell}$ there is an associated $\Bbb{P}^1\to G/B$ which sends $t\in \Bbb{A}^1$ to $\exp(tE_{\beta})s_{\beta}\dot{e}\in G/B$ and $t=\infty$ to $\dot{e}$. In fact the entire Schubert cell $Bs_{\beta}B/B$ is the image of $\Bbb{A}^1$, and the degree of the line bundle $\mathcal{L}_{\omega_{k}}$ along this curve is $\delta_{k,\ell}$, and so  $\mathcal{L}_{\omega_{\ell}}$ is dual to this curve in the Chow group of $G/B$ (in the Schubert basis).

To prove Proposition \ref{culture}, we first intersect $S_u$ with $\left[\overline {Bs_{\alpha_{\ell}}B/B}\right]\times  g[X_w]$ for a general $g\in G$, and $w\in W^P$ arbitrary such that $\ell(w)+1=\dim G/P-\ell(u)$. This shows that
\begin{itemize}
\item The intersection number $\mathcal{I}$ of $\beta_{\ell}$ and $[X_w]$ will equal the number of points of the form $(t,\bar{z})$ where (note that $G/B\times G/P$ has a transitive action of a group, and we can use general position arguments):
\begin{enumerate}
\item[(a)] $\bar{z}\in t X_u\in G/P$ (since ($(t,\bar{z})\in S_u$),
\item[(b)] $t\in {Bs_{\alpha_{\ell}}B/B}=\Bbb{A}^1$, and
\item[(c)] $\bar{z}\in X_w$.
\end{enumerate}
\end{itemize}

If ${\alpha_{\ell}}\in \Delta(Q_u)$, then $tX_u$ does not vary with $t$ (since $s_{\alpha_{\ell}}X_u\subseteq X_u$), and by general position arguments, the intersection number $\mathcal{I}$ is zero. Therefore if $u^{-1}\alpha_{\ell}\in R_{\mathfrak{l}}^+$ or  $u^{-1}\alpha_{\ell}\in R^-$, then
$tX_u$ does not vary with $t$. Therefore unless  $u^{-1}\alpha_{\ell}\in R^+ - R_{\mathfrak{l}}^+$, the intersection number $\mathcal{I}$ is zero (independently of $w$).
The following Lemma therefore shows the first part of Proposition \ref{culture}.
\begin{lemma}
The following are equivalent:
\begin{enumerate}
\item $u^{-1}\alpha_{\ell}\in R^+ - R_{\mathfrak{l}}^+$;
\item $\hat{u}_{\ell}\in W^P$ and $u\leto{\alpha_{\ell}}\hat{u}_{\ell}$.
\end{enumerate}
\end{lemma}
\begin{proof} Assume (1). We first show that $\hat{u}_{\ell}\in W^P$. We need to show that $\hat{u}_{\ell} R^+_{\frl}\subset R_+$. Assume the contrary. Now $\hat{u}_{\ell}R^+_{\frl}=s_{\alpha_{\ell}} u R^+_{\frl}\subseteq s_{\alpha_{\ell}}R^+$.
The only positive root which $s_{\alpha_{\ell}}$ takes to a negative root is $\alpha_{\ell}$, so we will have $\alpha_{\ell}\in u R^+_{\frl}$ which contradicts our assumptions.  Therefore we have shown that $\hat{u}_{\ell}\in W^P$. From $u^{-1}\alpha_{\ell}\in R^+$, we get $\ell(\hat{u}_{\ell})\geq \ell(u)+1$, which should be an equality since $\hat{u}_{\ell}=s_{\alpha_{\ell}}u$. Therefore (2) holds.

Now assume (2). The length condition in $u\leto{\alpha_{\ell}}\hat{u}_{\ell}$ implies that $u^{-1}\alpha_{\ell}\in R^+$.
If $u^{-1}\alpha_{\ell}\in R_{\mathfrak{l}}^+$, then $\alpha_{\ell}\in u R_{\mathfrak{l}}^+$, then $\hat{u}_{\ell}R_{\mathfrak{l}}^+$
contains $-\alpha_{\ell}$ which is a negative root. This  contradicts $\hat{u}_{\ell}\in W^P$.
\end{proof}

Now assume  $\hat{u}_{\ell}\in W^P$ and $u\leto{\alpha_{\ell}}\hat{u}_{\ell}$. We need to show that $\beta_{\ell}=[X_{\hat{u}_{\ell}}]$. We will show that the intersection number $\mathcal{I}$ of $\beta_{\ell}$ and $[X_w]$ is the same as the intersection number of $[X_{\hat{u}_{\ell}}]$ and $[X_w]$. This will finish the proof of Proposition \ref{culture}, since $w$ was arbitrary.

The intersection number $\mathcal{I}$  is the count of pairs $(\bar{z},t)$ satisfying conditions (a), (b), (c) above.
By Lemma \ref{ash}, the sets $tC_u$ are distinct and have $C_{\hat{u}_{\ell}}$ for their union. Therefore $\mathcal{I}$ equals  the intersection number of $X_{\hat{u}_{\ell}}$ and $gX_w$, as desired (we can assume that the intersection takes place in the open Schubert cells in each by dimension counting).

\subsection{Proof of Theorem \ref{timeticks}}
Let $(u_1,\dots,u_s)$ be as in Proposition 1.7. Recall that
$$\sum_{i=1}^s (\dim G/P-\ell(u_i)) =\dim G/P +1.$$
Let  $\widetilde{D}(j,v)\subseteq (G/B)^s\times G/P$ be as defined in Definition \ref{gradechange}. Note that ${D}(j,v)=\pi_*(\widetilde{D}(j,v))$ (use the fact that $\widetilde{D}(j,v)$ is not contained in $\mathcal{R}$ as proved in Section \ref{refer}).

We have $s$ morphisms $p_i: (G/B)^s\times G/P \to (G/B) \times (G/P)$. The scheme theoretic intersection of
$p_i^{-1}S_{u_i}$ equals $\widetilde{D}(j,v))$. This intersection is proper because the codimension of $\widetilde{D}(j,v)$ in $(G/B)^s\times G/P$ is the sum of  codimensions of $X_{u_i}$. The cycle class of
$\widetilde{D}(j,v))$ is therefore the cup product of the pull backs of cycle classes of $S_{u_i}$. Theorem
\ref{timeticks} now follows from ${D}(j,v)=\pi_*(\widetilde{D}(j,v))$.
\begin{remark}\label{ramifyy}
Suppose we consider a codimension one Schubert cell $C_{v}\subseteq X_{w_j}$ with $v\leto{\beta}w_j$, and $\beta$ not simple, $v\in W^P$. Define $u_1,\dots,u_s$ as in Equation \eqref{setU}, and let $D$ be the right hand side of \eqref{defDD}. Let $\widetilde{D}\subseteq \mathcal{Z}$ be the right hand side of \eqref{gradechange}. Then by \cite[Proposition 8.1]{BKR}, $\widetilde{D}$ lies in the closure of $\mathcal{R}$, and hence $D$, the image of $\widetilde{D}$ is of codimension $\geq 2$ in $(G/B)^s$. The element $\pi_*(\widetilde{D})\in A^1((G/B)^s)$ is zero, and the formulas of Theorem \ref{timeticks} apply also in this case. Therefore one gets vanishing of several intersection numbers.
\end{remark}
\begin{remark}
The proof of Theorem \ref{timeticks} shows that one obtains formulas for a divisor class supported on
the locus (with possible multiplicities) given by the right side of \eqref{defDD} (as a suitable push forward)
for arbitrary $u_1,\dots,u_s$ satisfying $\sum (\dim(G/P)-\ell(u_i))=\dim G/P$ and $u_i\in W^P$. This divisor class is zero if and only if the right side of \eqref{defDD} is not codimension one in $(G/B)^s$ (it is always irreducible).
\end{remark}
\section{Faces of the eigencone}\label{regular}

\subsection{The face $\mf$ as a product}
Our aim in this section is to prove Theorem \ref{conebij}. Consider the map \eqref{until}. We first show that it is an injection (actually we prove a stronger statement with $\Bbb{Q}$ coefficients rather than $\Bbb{Q}_{\geq 0}$ coefficients). Suppose
\begin{equation}\label{coffee}
\sum a_b\delta_b +f=\sum a'_b\delta_b +f',\ a_b,a'_b\in \Bbb{Q}.
\end{equation}
It suffices to show $a_b=a'_b$ for all $b$. Fix $b$ and suppose $\delta_b=[D(j,v)]$ where $v\leto{\alpha_{\ell}}w_j$. Then we may apply $\alpha_{\ell}$ to the  $j$th coordinate of  \eqref{coffee}. This gives $a_b=a'_b$ by using Corollary \ref{correspondence2} below.

The following is an easy corollary of Theorem \ref{timeticks}.
\begin{corollary}\label{correspondence2}
Consider a  pair $(j,v)$ with $v\leto{\alpha_{\ell}}w_j$, and set
$$\mathcal{O}(D(j,v))=\ml_{\lambda_1}\boxtimes \ml_{\lambda_2}\boxtimes\dots\boxtimes \ml_{\lambda_s}\in \Pic(G/B)^s.$$
Then,
\begin{enumerate}
\item $\lambda_{j}(\alpha_{\ell}^{\vee})=1$, and hence $\alpha_{\ell}(\kappa(\lambda_j))>0$.
\item Suppose $(j',v')\neq (j,v)$ with  $v'\leto{\alpha_{\ell'}}w_{j'}$.
Then, $\lambda_{j'}(\alpha_{\ell'}^{\vee})=0$, and hence $\alpha_{\ell'}(\kappa(\lambda_{j'}))=0$.
\end{enumerate}
\end{corollary}
\begin{proof}
Set $$
u_i = \left\{
\begin{array}{cc}
 w_i, & i\ne j\\
 v, & i=j
\end{array}\right.
$$
Using Theorem \ref{timeticks}, for (1), the coefficient $c_{j,\ell}=\lambda_{j}(\alpha_{\ell}^{\vee})$ is just the multiplicity in the intersection product \eqref{kishi} in ordinary cohomology product which is one, since it is one in the deformed product $\odot_0$ by assumption.

For (2), consider the case $j=j'$ first: We start by showing that $v^{-1}\alpha_{\ell'}$ is not a positive root. Now $v^{-1}\alpha_{\ell'}=w^{-1}(s_{\alpha_{\ell}}\alpha_{\ell'})=w^{-1}(\alpha_{\ell'}+m\alpha_{\ell})$ with $m\ge 0$ since $\ell\neq \ell'$. Now both $w^{-1}(\alpha_{\ell})$ and $w^{-1}(\alpha_{\ell'})$ are negative roots by assumption and hence  $v^{-1}\alpha_{\ell'}$  is not in $R^+$, and hence $\lambda_{j}(\alpha_{\ell'}^{\vee})=0$ using Theorem\ref{timeticks}.

If $j\neq j'$ then we need to show that $w_{j'}^{-1}\alpha_{\ell'}$ is not a positive root, which follows from $v'\leto{\alpha_{\ell'}}w_{j'}$.
\end{proof}

The surjection part of Theorem \ref{conebij} follows from Corollary \ref{correspondence2},
and the following:
\begin{proposition}\label{egale}
Suppose
$$\ml=\ml_{\mu_1}\boxtimes\ml_{\mu_2}\boxtimes\dots\boxtimes\ml_{\mu_s}\in \Pic(G/B)^s$$
and $x=(\kappa(\mu_1),\dots,\kappa(\mu_s))$. Let $(j,v)$  with $v\leto{\alpha_{\ell}}w_j$ and $v\in W^P$.
Assume
\begin{enumerate}
\item  $H^0((G/B)^s,\ml)^G\neq 0$. Assume also that $x\in \mf$.
\item $\alpha_{\ell}(\kappa(\mu_j))>0$, i.e., $\mu_j(\alpha_{\ell}^{\vee})>0$.
\end{enumerate}
Let $m=\mu_j({\alpha_{\ell}}^{\vee})\in \Bbb{Z}_{>0}$, and
$$\ml'=\ml(-mD(j,v))=\ml_{\mu'_1}\boxtimes\ml_{\mu'_2}\boxtimes\dots\boxtimes\ml_{\mu'_s}\in \Pic(G/B)^s$$
and $x'=(\kappa(\mu'_1),\dots,\kappa(\mu'_s))$.
Then
\begin{enumerate}
\item $H^0((G/B)^s,\ml')^G\neq 0$. Thus all $\mu'_i$ are dominant and  $x'\in \mf$.
\item $\mu_j'({\alpha_{\ell}}^{\vee})=0$.
\end{enumerate}
\end{proposition}
\begin{proof}
Start with a non-zero invariant section  $s\in H^0((G/B)^s,\ml)^G$. We will show that $s$ vanishes on $D(j,v)$:
This will show that $\ml(-D(j,v))$ has invariant sections and lies on $\mf$ (also use Theorem \ref{Tone} (d)). Writing
$$\ml(-D(j,v))=\ml_{\nu_1}\boxtimes\ml_{\nu_2}\boxtimes\dots\boxtimes\ml_{\nu_s}\in \Pic(G/B)^s$$
we see using Corollary \ref{correspondence2}  that $\nu_j({\alpha_{\ell}}^{\vee})=\mu_j({\alpha_{\ell}}^{\vee})-1$, and we can iterate this procedure to get the desired result.

For the vanishing of $s$ on $D(j,v)$, start with a general point
 $x=(\bar{g}_1,\bar{g}_2,\dots,\bar{g}_s)\in D(j,v)$.
 Applying the considerations of Section \ref{necessary} below,
 set $$
u_i = \left\{
\begin{array}{cc}
 w_i, & i\ne j\\
 v, & i=j
\end{array}\right..
$$
We will show that inequality \eqref{inegalite} below fails: i.e., show that for a suitable $\alpha_{k}\not\in\Delta(P)$,
\begin{equation}\label{inequity}
\sum_{i=1}^s u_i^{-1}\mu_i(x_{k})>0
\end{equation}
However, we know that the point $x$ is on the face $\mf$, and hence
\begin{equation}
\sum_{i=1}^s w_i^{-1}\mu_i(x_{k})=0.
\end{equation}

Therefore it suffices to show that $(w_j^{-1}\mu_j - v^{-1}\mu_j)(x_{k})\leq 0$ for some  $\alpha_{k}\not\in\Delta(P)$, with a strict inequality for at least one $\alpha_{k}\not\in\Delta(P)$. Now
\begin{equation}\label{carnatic3}
w_j^{-1}\mu_j-v^{-1}\mu_j=v^{-1}(s_{\alpha_{\ell}}\mu_j-\mu_j) =-\mu_j({\alpha_{\ell}}^{\vee})v^{-1}\alpha_{\ell}
\end{equation}
By assumption $\mu_j({\alpha_{\ell}}^{\vee})>0$. Also we know $\beta=v^{-1}\alpha_{\ell}\in R^+$. Therefore the inequality holds. We now show that at least one inequality holds strictly.

We claim that  $\beta=v^{-1}\alpha_{\ell}\not\in R^+_{\mathfrak{l}}$ because if $\beta\in R^+_{\mathfrak{l}}$, then
$\alpha_{\ell}=v\beta$, and $-\alpha_{\ell}=s_{\alpha_{\ell}}v\beta=w_j\beta$, but $w_j\beta$ is a positive root since $w\in W^P$.

Therefore in the expression of the positive root $\beta$ as a sum of simple roots, at least one root $\alpha_{k}\in \Delta-\Delta(P)$ appears with a non-zero coefficient, and $v^{-1}\alpha_{\ell}(x_{k})>0$. For this $k$, by \eqref{carnatic3}, $(w_j^{-1}\mu_j - v^{-1}\mu_j)(x_{k})<0$, as desired.
\end{proof}

\subsection{Necessary inequalities}\label{necessary}
Suppose
\begin{enumerate}
\item $x=(\bar{g}_1,\bar{g}_2,\dots,\bar{g}_s)$ is an arbitrary point of $(G/B)^s$,
\item  $s\in H^0((G/B)^s,\ml_{\mu_1}\boxtimes\dots\boxtimes\ml_{\mu_s})^G$ with $s(x)\neq 0$. Assume further that
\item $\cap g_i C_{u_i}\neq \emptyset\subseteq G/P$. Here $P$ is a standard parabolic of $G$.
\end{enumerate}
We want to recall the (standard) proof of
\begin{equation}\label{inegalite}
\sum_{i=1}^s u_i^{-1}\mu_i(x_k)\leq 0
\end{equation}
whenever $\alpha_k\in \Delta-\Delta(P)$, under these conditions.

First assume that $\dot{e}\in\cap g_i C_{u_i}$ by translations in $G$. Next write down equations
$g_iu_i=p_i$ or $g_i=p_i u_i^{-1}$. Consider the (rational) one parameter subgroup $t^{x_k}$
and the limit point
$$\lim_{t\to 0}t^{x_k}(\bar{g}_1,\bar{g}_2,\dots,\bar{g}_s)=(\bar{h}_1,\dots,\bar{h}_s)\in (G/B)^s.$$
The action of $t^{x_k}$ on the fiber of $\ml_{\mu_1}\boxtimes\dots\boxtimes\ml_{\mu_s}$ over the limit point (since this measures the order of vanishing of $s$ as $t\to 0$) should be $\le 0$.

The desired inequality \eqref{inegalite} follows from
\begin{lemma}
The action of the (rational) one parameter subgroup $t^{x_k}$ on the fiber of $\ml_{\bar{\mu}_i}$ at $\bar{h}_i$ is given by the exponent  $-\mu_i(u_i x_{k})$.
\end{lemma}
\begin{proof}
This is because $\bar{h}_i$ is of the form $l_i v_iu_i^{-1}$ where $l_i$ is in the Levi subgroup $L$ and $v_i$ is in the
unipotent radical and commutes with $t^{x_k}$ (start with $p_i=l_i v'_i$ and write $v'_i$ as a product of one parameter subgroups). Therefore we need to compute the action of $t^{x_{k}}$ on the fiber of $\ml_{\bar{\mu}_i}$
at $u_i^{-1}\dot{e}$, which is a standard computation.
\end{proof}
\subsection{Extremal rays lie on regular facets of the eigencone}
\begin{lemma}\label{extreme}
Suppose that $\vec r = \Bbb{Q}_{\ge 0}(h_1,\hdots,h_s)$ is an extremal ray of $\Gamma_{\Bbb{Q}}(s,K)$. Then $(h_1,\hdots,h_s)$ satisfies some inequality (\ref{fox}) with equality (and maximal parabolic $P$).
\end{lemma}
\begin{proof}
Assume not. By symmetry, we may assume $h_1,\hdots,h_m$ are all nonzero for some $1\le m\le s$, and that $h_{m+1},\hdots,h_s$ are all $0$. By definition there exist $k_1,\hdots,k_s\in \mathfrak{k}$ such that $C(h_j) = \overline{k_j} \in \mathfrak{k}/K$ and $\sum k_j = 0$. If $m=1$, then $k_1\ne 0$ but $k_j = 0$ for each $j> 1$, contradicting the sum condition. So $m$ is at least $2$.

Since each inequality (\ref{fox}) holds with strict inequality, $\vec r$ must be an extremal ray of $\frh_+^s$. However, the condition $h_1\in \frh_+$ is invariant under multiplication by $\Bbb{R}_{>0}$; thus for arbitrarily small values of $\epsilon>0$,
$$
((1+\epsilon)h_1,h_2,\hdots,h_s) \text{ and } ((1-\epsilon)h_1,h_2,\hdots,h_s)
$$
are elements of $\frh_+^s$ which are not proportional (since $h_2\ne 0$), but their sum gives the same ray $\vec r$. This contradicts the extremality of $\vec r$.
\end{proof}
\section{Building blocks for induction}

We need to find the remaining extremal rays on $\mf_{\Bbb{Q}}$, i.e., the extremal rays of $\mf_{2,\Bbb{Q}}$ under the killing form bijection of Proposition \ref{bij}, we are interested in the extremal rays of the $\Bbb{Q}$ cone generated by line bundles
$$\ml=\ml_{\lambda_1}\boxtimes \ml_{\lambda_2}\boxtimes\dots\boxtimes \ml_{\lambda_s}\in \Pic(G/B)^s$$
such that
\begin{enumerate}
\item $H^0((G/B)^s,\ml^N)^G \neq 0$ for some $N>0$
\item For each of the $q$ pairs $(j,v)$  with $v\leto{\alpha_{\ell}}w_j$ and $v\in W^P$, we have $\lambda_j(\alpha_{\ell}^{\vee})=0$
\end{enumerate}
We want to replace $(G/B)^s$ by a product of partial flag varieties $\prod_{i=1}^s G/Q'_{w_i}$ so that line bundles
on the latter pull back to line bundles $$\ml=\ml_{\lambda_1}\boxtimes\dots\boxtimes \ml_{\lambda_s}$$ on $(G/B)^s$ which satisfy all the linear equalities required in (2) above.

\begin{defi}
For $w\in W^P$, define
 $$\Delta'_w=\{\alpha\in \Delta\mid s_{\alpha} w< w\}=\Delta\cap wR^-\subseteq \Delta_w$$
and let ${Q}'_w\subseteq Q_w$ be the corresponding standard parabolic subgroup.
\end{defi}

\begin{lemma}\label{medic}
Let $\ml_{\lambda}$ be the pullback to $G/B$ of a line bundle on $G/Q'_w$, and  $v\leto{\alpha_{\ell}} w$, $v,w\in W^P$. Then,
$\lambda({\alpha_{\ell}}^{\vee})=0$.
\end{lemma}
\begin{proof}
This follows from $\alpha_{\ell}\in \Delta'_w$ since $w^{-1}\alpha_{\ell}\in R^-$.
\end{proof}

The group $w^{-1}Bw\cap P$ played a key role in various constructions in the previous sections. It was important in those arguments that it mapped onto $B_L$ under the projection to $L$. The group $Q'_w$ has the same property (but not $Q_w$). The following proof was communicated to us by S. Kumar:
\begin{lemma}\label{sk}
$w^{-1}{Q}'_w w\cap L =B_L.$
\end{lemma}
\begin{proof}
The inclusion $w^{-1}Q'_w w\cap L \supset B_L$ is easy because $wB_Lw^{-1}\subset B\subseteq Q'_w$.
For the other direction, we are reduced to proving that (also see Lemma \ref{else}),
\begin{equation}\label{shew}
w^{-1} R(Q'_w)\cap R^{-}(L)=\emptyset
\end{equation}
 Pick a $w^{-1}\gamma=\beta$ in the intersection. Clearly $\gamma$ is a negative root, since $w\beta$ is a negative root (use $wR^{-}(L)\subset R^-$). Write $\gamma=-\sum \gamma_i$ with $\gamma_i\in \Delta'_w$ and simple. Now $w^{-1}\gamma_i$ are negative, and hence $w^{-1}\gamma$ is a positive root, a contradiction.
\end{proof}
\subsection{Enlargement of Schubert cells}
Let $C'_w\subseteq  Y_w\subseteq X_w$ be the open $Q'_w$ orbit in $X_w$.
The following lemma relates $C'_w$ to simple codimension one Schubert cells in $X_w$ (see Lemma \ref{Roma}).
\begin{lemma}\label{outside}
Let $v\leto{\beta} w$ with $v,w\in W^P$. The following are equivalent:
\begin{enumerate}
\item $C_v\subseteq Y_w$.
\item $\beta\in \Delta'_w$ (in particular, $\beta$ is a simple root).
\item $C_v\subseteq C'_w$.
\end{enumerate}
Therefore $Y_w-C'_w$ is codimension $\geq 2$ in $Y_w$.
\end{lemma}
\begin{proof}
If $C_v\subseteq Y_w$ then $\beta\in \Delta_w=\Delta\cap w(R_\frl^+\sqcup R^-)$, and hence $\beta\in \Delta$.  We will also have $w^{-1}\beta\in R^-$, and hence $\beta\in \Delta'_w=\Delta\cap wR^{-}$.  Therefore (1) implies (2).

Under the assumption (2), $s_{\beta}\in Q'_w$, which implies $s_\beta \dot w\in C'_w$ and hence $C_v\subseteq C'_w$. Therefore (2) implies (3). Since $C'_w\subseteq Y_w$, it follows that (3) implies (1).
\end{proof}
\subsection{Universal families over $\prod_{i=1}^s (G/Q'_{w_i})$}\label{wsmm}

Define the universal intersection locus
$$
\mathcal{X}' = \left\{(\bar g_1,\hdots\bar g_s,\bar{z})\in \prod_{i=1}^s (G/Q'_{w_i} \times G/P )\mid \bar{z}\in g_iX_{w_i}~\forall i\right\},
$$
and similarly define subloci $\mathcal{Z}'\supseteq\mathcal{C}'$ using the $Z_{w_i}, C'_{w_i}$, respectively, in place of the $X_{w_i}$. (For the scheme structures, see \cite{BKR}.) As shown in \cite{BKR}, the projection map $\pi':\mathcal{X}' \to\prod_{i=1}^s (G/Q'_{w_i})$ is birational.
Moreover, because the $X_w$ are all normal, each $X_{w_i}\setminus Z_{w_i}$ has codimension $\ge 2$ in $X_{w_i}$; accordingly, $\mathcal{X}'\setminus\mathcal{Z}'$ has codimension $\ge 2$ in $\mathcal{X}'$. Let $\pi':\mathcal{X}'\to \prod_{i=1}^s (G/Q'_{w_i})$.

\begin{lemma}\label{codimensioon}
$\overline{\pi'(\mathcal{X}'-\mc')}\subseteq \prod_{i=1}^s (G/Q'_{w_i})$ is of codimension $\geq 2$.
\end{lemma}
\begin{proof}
It suffices to show that ${\pi'(\mathcal{X}'-\mc')}\subseteq \prod_{i=1}^s (G/Q'_{w_i})$ is of codimension $\geq 2$.
Consider the fiber product diagram
\begin{equation}\label{diagrammo3}
\xymatrix{
\mathcal{X}\ar[r]^{\tilde{\phi}}\ar[d]^{\pi} & \mathcal{X}'\ar[d]^{\pi'} \\
(G/B)^s\ar[r]^{\phi} & \prod_{i=1}^s(G/Q'_{w_i})}
\end{equation}
Now $\phi$ is a smooth fiber bundle over a smooth base, and it suffices to show that $\phi^{-1}(\pi'(\mathcal{X}'-\mc'))$ is of codimension $\geq 2$ in $(G/B)^s$. We have (see the remark below)
\begin{equation}\label{fiber}
\phi^{-1}({\pi'(\mathcal{X}'-\mc')}))={\pi({\tilde{\phi}^{-1}(\mathcal{X}'-\mc')})}
\end{equation}
and  $\tilde{\phi}^{-1}(\mathcal{X}'-\mathcal{C}')\supseteq \mathcal{X}-\mathcal{Y}$ with complement of codimension $\geq 2$ in $\mathcal{X}$ (by Lemma \ref{outside}, note  also that $\mathcal{X}-\mathcal{Z}$ is codimension $\geq 2$ in $\mathcal{X}$). The desired statement follows from \cite[Proposition 8.1]{BKR}, which shows that up to codimension two, $\mathcal{Z}-\mathcal{Y}$ lies in the ramification divisor of $\pi$.
\begin{remark}
The equality \eqref{fiber} can be verified (analytic) locally on $(G/B)^s$. Let $\mathcal{U}$ be an open subset of $\prod_{i=1}^s(G/Q'_{w_i})$, such that $\mathcal{U}'=\phi^{-1}(\mathcal{U})=\mathcal{U}\times \Lambda$ for a suitable $\Lambda$. Then if $\Gamma=(\mathcal{X}'-\mathcal{C}')\cap \pi'^{-1}\mathcal{U}$  then the left hand side looks (over $\mathcal{U}'$) like ${\pi'(\Gamma)}\times \Lambda$
and the right hand side like $\pi(\Gamma\times \Lambda)$.
%Here $\pi'$ is an arbitrary map and $\pi$ is projection on to first factor followed by $\pi'$.
\end{remark}

\end{proof}

\section{Parameter spaces for induction}

\begin{remark}
The constructions of Section \ref{parallel} can be generalized: let $\psi:P\to L=P/U$ be the quotient map, and let $M_1,\hdots,M_s$ be standard parabolic subgroups satisfying
$$
\psi\left(w_i^{-1}M_iw_i \right) = B_L
$$
for each $i=1,\hdots,s$. For brevity, write $\vec M = (M_1,\hdots,M_s)$. Then one can define $\op{Fl}_G(\vec M) = \prod_{i=1}^s (G/M_i)/G$, which parameterizes $G$-spaces $E$ together with elements $\bar{g}_i\in E/M_i, i=1,\dots,s$, and $\mathcal{C}(\vec M)$, which parameterizes principal $G$-spaces $E$ together with elements $\bar{g}_i\in E/M_i$ and a single element $\bar{z}\in E/P$ so that $\bar z\in g_i M_iw_iP$ for all $i$.

Analogues of Lemma \ref{zusammen}, the diagram in Section \ref{gram}, the Levification process, Proposition \ref{comparison}, the ramification loci of Section \ref{ramify}, and Corollary \ref{carol} all exist/hold in this general setting. The particular case in which we are now interested is where $M_i = Q'_{w_i}$ for each $i$, as we now convey in detail.

\end{remark}

We introduce the following stacks:
\begin{defi}
The stack $\op{Fl}'_G$ parameterizes principal $G$-spaces $E$ together with elements $\bar{g}_i\in E/Q'_{w_i}$; i.e.,
$$\Fl'_G=\prod_{i=1}^s (G/Q'_{w_i})/G.$$
The spaces $\op{Fl}_L$ are left unchanged in the induction operation.
\end{defi}

\begin{defi}
The moduli stack $\hat{\mathcal{C}'}$ parameterizes principal $G$-spaces $E$ together with elements $\bar{g}_i\in E/Q'_{w_i}$ and a single element $\bar{z}\in E/P$ so that $\bar{z}\in g_i C'_{w_i}$ for $i=1,\dots,s$.
There is a natural map $\pi':\hat{\mathcal{C}'}\to \op{Fl}'_G,$ and it is easy to see that $\hat{\mathcal{C}'}=\mathcal{C}'/G$ ($\mc'$ was defined in Section \ref{wsmm})
\end{defi}
Similar to Lemma \ref{zusammen}, we have the following
\begin{lemma}
The stack $\hat{\mc'}$ parameterizes principal $P$ bundles $E'$ together with $\bar{z}_i\in E'/(w_i^{-1}Q'_{w_i}w_i\cap P), i=1,\dots,s.$
\end{lemma}
This lemma, Lemma \ref{sk} and the map $P\to L$ give rise to ${\tau'}:\hat{C}'\to \Fl_L$. The map $L\to P$ gives a map $i': \Fl_L\to \hat{C}'$, and the map $\tilde{i}'=\pi'\circ i'$. Similar to \eqref{diagrammo} we have the following diagram
\begin{equation}\label{diagrammo2}
\xymatrix{
\mathcal{C}\ar[r]\ar[d]& \mathcal{C}'\ar[r]\ar[d]^{\pi'} & \hat{\mathcal{C}'}\ar[d]^{\pi'}\ar[dr]_{{\tau'}}  \\
(G/B)^s\ar[r] & \prod_{i=1}^s(G/Q'_{w_i})\ar[r] & \op{Fl}'_G & \op{Fl}_L\ar@/_/[ul]_{i'}\ar[l]^{\tilde{i}'}}
\end{equation}
\subsection{Levification in the new setting}
$\hat{\mc'}$ retracts to $\Fl_L$ by  Levification, generalising the constructions in Section \ref{secLev}: Let  $\phi_t:P\to P$ be as in Section \ref{secLev}.

\begin{defi}\label{levity2}
Let $(E',\bar{z}_1,\dots,\bar{z}_s)$ be a point of $\hat{\mc'}$ where $E'$ is a principal $P$-space and $\bar{z}_i\in
E'/(w_i^{-1}Q'_{w_i}w_i)\cap P$. Define the Levification family $E_t'=E'\times_{\phi_t} P$ for $t\in \Bbb{A}^1$, and $\bar{z}_i(t)=\bar{z}_i\times_ {\phi_t} e$. Clearly at $t=0$, $(E_t,\bar{z}_1(t),\dots,\bar{z}_s(t))$ is in the image of $i:\Fl_L\to \hat{\mc}$.
\end{defi}
Proposition \ref{comparison} generalises (with the same proof) to this new setting with ${\tau'}:\hat{\mc'}\to\Fl_L$.
\begin{proposition}\label{comparison2}
Let $U$ be a non-empty open substack of $\Fl_L$,  $\ml$ be a line bundle on ${\tau'}^{-1}(U)$ and $\mathcal{M}=i'^*\mathcal{L}$, a line bundle on $U$, where $i':\Fl_L\to \hat{\mc}$.
Then
\begin{enumerate}
\item $\ml={\tau'}^*\mathcal{M}$. Therefore $\tau'^*$ and $i'^*$ set up isomorphisms $\Pic(U)\leto{\sim}\Pic(\tau'^{-1}(U))$.
\item If $\gamma_{\mathcal{M}}:Z^0(L)\to \Bbb{C}^*$ is trivial then
$i'^*:H^0(\tau'^{-1}(U),\ml)\to H^0(U,\mathcal{M})$ is an isomorphism.
\end{enumerate}
\end{proposition}
\subsection{Ramification divisors in the new setting}
Let $\mathcal{R}'$ be the ramification divisor of the map $\pi':\mathcal{C}'\to \prod_{i=1}^s(G/Q'_{w_i})$. Similar to Section \ref{ramify},
The line bundle $\mathcal{O}(\mathcal{R}')$ is the pull back of a natural line bundle on $\hat{\mc'}$: Consider a
$P$ bundle $E'$ together with $\bar{z}_i\in E'/(w_i^{-1}Q'_{w_i}w_i\cap P)$. Consider the map of vector spaces
\begin{equation}\label{eseq2}
E'\times_P T(G/P)_{\dot e}\to \bigoplus \frac{E'\times_P T(G/P)_{\dot{e}}}{z_i\times_P T(w_i^{-1}C'_{w_i})_{\dot{e}}}
\end{equation}
The determinant line and theta section for this map give a line bundle and a section on  $\hat{\mc'}$, denoted
by $\mathcal{O}(\hat{\mathcal{R}'})$ and  $\hat{\mathcal{R}'}$.

As in Section \ref{ramify},  $\mathcal{O}(\hat{\mathcal{R}'})$ and  $\hat{\mathcal{R}'}$ pull back to $\mathcal{O}(\mathcal{R}')$ and the divisor $\mathcal{R}'$ on $\mathcal{C}'$. Pulling back $\mathcal{O}(\hat{\mathcal{R}'})$ and  $\hat{\mathcal{R}'}$ via $i'$ we get the same divisor $\mathcal{R}_L$  and line bundle $\mathcal{O}(\mathcal{R}_L)$ on $\Fl_L$ as in Section \ref{ramify} (by looking at the complex computing the pull back determinant line, for example).

Similar to  Corollary \ref{carol},
\begin{corollary}\label{carol2}
\begin{enumerate}
\item $\mathcal{O}(\hat{\mathcal{R}'})={\tau'}^{*} \mathcal{O}(\mathcal{R}_L)$
 \item $\hat{\mathcal{R'}}={\tau'}^{-1} \mathcal{R}_L$ so that  $i':\Fl_L-\mathcal{R}_L\to \hat{\mathcal{C}'}-\hat{\mathcal{R}'}$.
\end{enumerate}
\end{corollary}

The restricted flag setting has one new feature, which follows from Lemma \ref{codimensioon} and Zariski's main theorem:
\begin{lemma}\label{codim2}
$\pi':\mathcal{C}'-\mathcal{R}'\to \prod_{i=1}^s(G/Q'_{w_i})$ is an open immersion whose complement has codimension $\geq 2$.
\end{lemma}
\section{Picard groups}
\begin{corollary}\label{restricto}
$\pi'^*:\Pic(\Fl_G')\to \Pic(\hat{\mc}'-\hat{\mathcal{R}'})$ is an isomorphism.
\end{corollary}
\begin{proof}
We need to compare the set of $G$-equivariant line bundles on $\prod_{i=1}^s(G/Q'_{w_i})$ and on the
open subset $U=\mathcal{C}'-\mathcal{R}'$. If a $G$ equivariant line bundle on $\prod_{i=1}^s(G/Q'_{w_i})$ becomes trivial on $U$, then it is trivial as a line bundle on $\prod_{i=1}^s(G/Q'_{w_i})$ (by Lemma \ref{codim2}) and hence trivial as a $G$-bundle.

To show surjectivity, extend a $G$-equivariant line bundle $\ml$ on $U$ first as a line bundle to
all of $\prod_{i=1}^s(G/Q'_{w_i})$. We have isomorphisms $\ml\to\phi_g^*\ml$ on $U$ (here $\phi_g$ is
the action of $G$ on $\prod_{i=1}^s(G/Q'_{w_i})$). These actions extend to all of $\prod_{i=1}^s(G/Q'_{w_i})$, since a section of the line bundle $\Hom(\ml,\phi_g^*\ml)$ on $U$ will extend to the whole space by codimension considerations.
\end{proof}
\begin{defi}\label{doc}
Let $U$ be any non-empty open substack of $\Fl_L$. Let $\Pic^{\deg=0}(U)\subset \Pic(U)$ denote the subgroup of line bundles $\mathcal{M}$ with $\gamma_{\mathcal{M}}$ trivial (i.e., center of $L$ acts trivially).
\begin{enumerate}
\item Let $\Pic^{\deg =0}(\Fl_G)$ denote the subgroup of line bundles whose pull back under $\tilde{i}$ is in
$\Pic^{\deg=0}(\Fl_L)$.
\item Let $\Pic^{\deg =0}(\Fl'_G)$ denote the subgroup of line bundles whose pull backs under the natural map $\Fl_G\to \Fl'_G$ are in $\Pic^{\deg=0}(\Fl_G)$.
%$=\Pic(\Fl'_G)\cap \Pic^{\deg =0}(\Fl_G)$.
\item Finally, $\Fl_{L^{\op{ss}}}$ denotes the stack parametrizing principal $L^{\op{ss}}$-spaces $F$ together with elements $\bar{q}_i\in F/B_{L^{\op{ss}}}$; i.e., $\Fl_{L^{\op{ss}}} = \left(L^{\op{ss}}/B_{L^{\op{ss}}}\right)^s/L^{\op{ss}}$.
\end{enumerate}
\end{defi}

\subsection{Picard groups for the Levi subgroup}
\begin{lemma}\label{lemmE}
\begin{enumerate}
\item $\Pic^L(L/B_L)=\op{Hom}(T,\Bbb{C}^*)=\Pic(G/B)=\Pic^G(G/B)$.
\item There is a surjection $(\Pic^L(L/B_L))^s\to \Pic(\Fl_L)$ whose kernel is the set of tuples
$(\mu_1,\dots,\mu_s)$ (using (1)) such that $\mu_i$ are trivial on $T(L^{\op{ss}})$, the maximal torus of $L^{\op{ss}}$, and $\sum \mu_i=0$ (i.e.,  given the triviality of $\mu_i$ on $T(L^{\op{ss}})$, equivalent to $\sum \mu_i(x_k)=0$ for $\alpha_k\not\in\Delta(P)$).
\end{enumerate}
\end{lemma}
\begin{proof}
For (1), note that $L$ equivariant line bundles on $L/B_L$ are in one-one correspondence with characters of $B_L$, which coincide with characters of $T$.

For (2), $\Pic(\Fl_L)$ is the set of (diagonal) $L$-equivariant bundles on $(L/B_L)^s$. Therefore there is a map
$(\Pic^L(L/B_L))^s\to \Pic(\Fl_L)$. This is surjective because an $L$-equivariant line bundle on $(L/B_L)^s$, as a line bundle on $(L^{\op{ss}}/B_{L^{\op{ss}}})^s$, is of the form $\ml=\ml_{\mu_1}\boxtimes\dots \boxtimes \ml_{\mu_s}$, where $\mu_i$ are characters of $T'$. These can be extended to characters of $T$ since $T=T'\times\left(\Bbb{C}^*\right)^a$, where $a$ corresponds to the number of simple coroots $\alpha_k^{\vee}, \alpha_k\not\in\Delta(P)$, and we can view
$\ml_{\mu_1}\boxtimes\dots \boxtimes \ml_{\mu_s}$ as an element of $(\Pic^L(L/B_L))^s$. This gives rise to a diagonal $L$-equivariant line bundle on $(L/B_L)^s$, call it $\ml'$. As line bundles on $(L/B_L)^s$, $\ml$ and $\ml'$ coincide, and therefore as equivariant line bundles, differ by a character $\lambda$ of $L$. We replace $\mu_1$ by $\mu_1+\lambda$ (and leave other $\mu_i$ unchanged) and then have $\ml=\ml'\in\Pic(\Fl_L)$.

The kernel of the map in (2) is identified similarly: A tuple $(\ml_{\mu_1},\dots,\ml_{\mu_s})$ which maps to zero gives the trivial line bundle on $(L^{\op{ss}}/B_{L^{\op{ss}}})^s$, hence $\mu_i$  are trivial restricted to $T(L^{\op{ss}})$. The center of $L$ should also under the diagonal action act trivially, so $\sum \mu_i$ is trivial on $Z^0(L)$. Now $L^{\op{ss}}$ and $Z^0(L)$
generate $L$, and we get (2).
\end{proof}

\subsection{Comparison of Picard groups of flag varieties for $L$ and for $L^{\op{ss}}$}
\begin{lemma}\label{threes}
\begin{enumerate}
\item The natural mapping (see Definition \ref{doc})
$\Pic^{\deg =0}(\Fl_L)\to \Pic(\Fl_{L^{\op{ss}}})$ is an injection of semigroups, which is an isomorphism
$\tensor \Bbb{Q}$:
\begin{equation}\label{tread}
\Pic^{\deg=0}_{\Bbb{Q}}(\Fl_L)\leto{\sim} \Pic_{\Bbb{Q}}(\Fl_{L^{\op{ss}}}).
\end{equation}
\item  $\Pic^{+}_{\Bbb{Q}}(\Fl_L)\subseteq \Pic^{\deg=0}_{\Bbb{Q}}(\Fl_L)$ and \eqref{tread}
gives a cone bijection
\begin{equation}\label{tread2}
\Pic^{+, \deg=0}_{\Bbb{Q}}(\Fl_L)\leto{\sim} \Pic^+_{\Bbb{Q}}(\Fl_{L^{\op{ss}}}).
\end{equation}
\end{enumerate}
\end{lemma}
\begin{proof}
The injection statement follows from Lemma \ref{lemmE}. For the surjection, given a line bundle $\ml'$ on $\Fl_{L^{\op{ss}}}$ we can find a line bundle $\ml$ on $\Fl_L$ which maps to  $\ml'$ under the natural restriction map. The action of $Z^0(L)$ may not be trivial, but we can tensor
$\ml$ by a line bundle of the form $\ml_{\lambda}\boxtimes\mathcal{O}\boxtimes\dots\boxtimes\mathcal{O}$ where $\lambda=\sum a_k\omega_{k}$, the sum taken over $\alpha_k\not\in \Delta(P)$ but possibly with $a_k\in \Bbb{Q}$, to make the action of $Z^0(L)$ trivial (note that $Z^0(L)\times L^{ss}\to L$ is an isogeny, so that any $\gamma$ in the dual of the Lie algebra of $Z^0(L)$ is the restriction of some element of $\frh_{\Bbb{Q}}^*$ that vanishes on the Lie algebra of $T(L^{ss})$). This proves (1).

It is also easy to check that the map $\Pic^{\deg=0}(\Fl_L)\to \Pic(\Fl_{L^{\op{ss}}})$ preserves global sections and (2) follows.

\end{proof}

\section{The induction operation}
Define
\begin{equation}\label{ind_isom}
\Ind=\Ind_L^G: \Pic(\Fl_L-\mathcal{R}_L)\leto{\sim} \Pic(\Fl'_G)
\end{equation}
as a composite of the isomorphism from Proposition \ref{comparison2}(1) applied to $U=\Fl_L-\mathcal{R}_L$ (and Corollary \ref{carol2}):
$$\Pic(\Fl_L-\mathcal{R}_L)\to \Pic(\hat{\mc}'-\hat{\mathcal{R}'})$$
and  the inverse of the isomorphism of Corollary \ref{restricto}. Note that $\left(\tilde{i}^{'}\right)^*\Ind_L^G(\mathcal{M})$ is isomorphic to $\mathcal{M}$ (here $\mathcal{M}\in \Pic(\Fl_L-\mathcal{R}_L)$).
\begin{lemma}\label{ones}
Using notation defined in Definition \ref{doc},
\begin{enumerate}
\item[(a)] The restriction mapping $\Pic(\Fl_L)\to \Pic(\Fl_L-\mathcal{R}_L)$.
\item[(b)] The restriction mapping $\Pic^{\deg=0}(\Fl_L)\to \Pic^{\deg=0}(\Fl_L-\mathcal{R}_L)$ is surjective.
\item[(c)] The isomorphism  $\Ind: \Pic(\Fl_L-\mathcal{R}_L)\leto{\sim} \Pic(\Fl'_G)$ restricts to an isomorphism
$$\Ind: \Pic^{\deg=0}(\Fl_L-\mathcal{R}_L)\leto{\sim} \Pic^{\deg=0}(\Fl'_G).$$
\end{enumerate}
\end{lemma}
\begin{proof}
Part (c) follows from the definition and part (a) implies (b). For (a), we define a linear section as follows:  a line bundle $\mathcal{M}$ on $\Fl_L-\mathcal{R}_L$ inducts to a line bundle $\ml$ on $\Fl'_G$ which can then be restricted to $\Fl_L$ via $\tilde{i}$, this is the desired lift: via $\pi'$,  $\hat{\mc}'-\hat{\mathcal{R}}'$ is an open substack of $\Fl'_G$. Therefore the pull back of $\ml$ to $\hat{\mc}'-\hat{\mathcal{R}'}$ is isomorphic to ${\tau'}^*\ml$, and since $\tilde{i}'=\pi\circ i'$, and ${\tau'}\circ i'$ is the identity on $\Fl_L$, the result follows.
\end{proof}

Recall that Lemma \ref{lemmE} gives a surjection from $(\Pic^L(L/B_L))^s=\Pic(G/B)^s$ to $\Pic(\Fl_L)$. Therefore a tuple  $(\mu_1,\dots,\mu_s)$ of weights for $G$ gives rise to an element of $\Pic(\Fl_L)$. The following theorem
gives a formula for its image, under induction, in $\Pic(\Fl'_G)$.
\begin{theorem}\label{general_Induction}
The composite induction map
\begin{equation}\label{induct}
\Pic(\Fl_L)\twoheadrightarrow \Pic(\Fl_L-\mathcal{R}_L)\leto{\sim} \Pic(\Fl'_G)\subseteq \Pic(\Fl_G)
\end{equation}
takes $(\mu_1,\dots,\mu_s)$ to $(\lambda_1,\dots,\lambda_s)$ where
\begin{equation}\label{dinner}
(\lambda_1,\dots,\lambda_s)=(w_1\mu_1,w_2\mu_2,\dots, w_s\mu_s) -\sum_{j=1}^s\bigl(\sum'_{v} w_j\mu_j(\alpha_{\ell}^{\vee})\mathcal{O}(D(j,v))\bigr).
\end{equation}
Here the sum is over $v\in W^P$, $j=1,\dots,s$ and simple roots $\alpha_{\ell}$ such that $v\leto{\alpha_{\ell}}w_j$.
\end{theorem}

\subsection{The induction operation and global sections}
\begin{lemma}\label{c2}
If $\mathcal{M}\in  \Pic^{\deg=0}(\Fl_L-\mathcal{R}_L)$, then
\begin{equation}\label{vibh1}
H^0(\Fl'_G,\Ind(\mathcal{M}))\leto{\left(\tilde{i}^{'}\right)^*} H^0(\Fl_L-\mathcal{R}_L,\mathcal{M})
\end{equation}
is an isomorphism.
Furthermore,
\begin{equation}\label{vibh2}
H^0(\Fl'_G,\Ind(\mathcal{M}))\leto{\left(\tilde{i}^{'}\right)^*} H^0(\Fl_L,\tilde{i}'^*\Ind(\mathcal{M}))
\end{equation}
is also an isomorphism.
\end{lemma}
We note that \eqref{vibh2} also follows from results of Roth \cite{Roth}, and is therefore not new.
\begin{proof}
The first isomorphism \eqref{vibh1} follows from Lemma \ref{codim2} and Proposition \ref{comparison2}:
$$H^0(\Fl'_G,\Ind(\mathcal{M}))=H^0(\hat{\mc'}-\hat{\mathcal{R}'},\pi^*\Ind(\mathcal{M}))=H^0(\hat{\mc'}-\hat{\mathcal{R}'},{\tau'}^*\mathcal{M})=H^0(\Fl_L-\mathcal{R}_L,\mathcal{M}).$$
 \eqref{vibh1} factors through the inclusion  $H^0(\Fl_L,\tilde{i}'^*\Ind(\mathcal{M}))\subseteq H^0(\Fl_L-\mathcal{R}_L,\mathcal{M})$ which gives \eqref{vibh2}.
\end{proof}
\begin{remark}
Suppose $\mathcal{M}\in  \Pic^{\deg=0}(\Fl_L)$, then in general $\tilde{i}'^*\Ind(\mathcal{M})$ and $\mathcal{M}$
may be different. They are identified on $\Fl_L-\mathcal{R}_L$ (even without the condition of action on center
on $\mathcal{M}$).
\end{remark}
\begin{remark}
Taking $\mathcal{M}=\mathcal{O}$ in \eqref{vibh1}, we see that $h^0(\Fl_L-\mathcal{R}_L,\mathcal{O})=h^0(\Fl'_G,\mathcal{O})=1$, and hence we recover the generalization of Fulton's conjecture proved
in \cite{BKR} (this proof is not really a different proof).
\end{remark}
Now note that $\Pic^{+}(\Fl_L)\subseteq \Pic^{\deg=0}(\Fl_L)$.
\begin{lemma}\label{twos}
\begin{enumerate}
\item The restriction mapping $\Pic^+(\Fl_L)\to \Pic^+(\Fl_L-\mathcal{R}_L)$ is surjective, with a linear section.
\item The isomorphism \eqref{ind_isom} restricts to an isomorphism  $\Ind: \Pic^+(\Fl_L-\mathcal{R}_L)\leto{\sim} \Pic^{+,\deg=0}(\Fl'_G)$
\end{enumerate}
\end{lemma}
\begin{proof}
For (1), the lift is just $\tilde{i}^*\Ind(\mathcal{M})$ as in Lemma \ref{c2}, and (1) follows from \eqref{vibh1} and \eqref{vibh2}. (2) follows from \eqref{comparison2}, applied to $U=\Fl_L-\mathcal{R}_L$, and Lemma \ref{codim2}.
\end{proof}

Theorem \ref{general_Induction} has the following corollary
\begin{corollary}\label{InductionC}
The induction map \eqref{induct} restricts to a surjection
$$\Pic^+(\Fl_L)\twoheadrightarrow \Pic^+(\Fl_L-\mathcal{R}_L)\leto{\sim} \Pic^{\deg=0}(\Fl'_G),$$
\end{corollary}

\subsection{Proof of Theorem \ref{IndT}}
Under the identification of $\Pic_{\Bbb{Q}}(\Fl_G)$ with $\frh_{\Bbb{Q}}^s$, it is easy to see that
$\Pic_{\Bbb{Q}}(\Fl'_G)$ corresponds to tuples $(h_1,\dots,h_s)$ such that $\alpha_m(h_j)=0$ whenever $(j,v)$ is such that $v\leto{\alpha_m} w_j$, and $\Pic^{+,\deg=0}_{\Bbb{Q}}(\Fl'_G)$ to the face $\mf_{2,\Bbb{Q}}$.
Similarly, $\Pic^+_{\Bbb{Q}}(\Fl_{L^{\op{ss}}})=\Gamma(s,K(L^{\op{ss}}))$ under the Killing form isomorphism $\frh_{L^{\op{ss}}}^*\to\frh_{L^{\op{ss}}}$ induced from $G$.

Theorem  \ref{IndT} follows immediately from Corollary \ref{InductionC} and the following lemma,
\begin{lemma}\label{lisse}
The Killing form isomorphism takes $\frh_{L^{\op{ss}}}\subseteq \frh$ to $(\frh^*)^{\deg=0}$, the set of $\lambda$ such that
$\lambda(x_k)=0$ for all $\alpha_k\not\in \Delta(P)$.
\end{lemma}
\begin{proof}
$\lambda(x_k)=0$ for all $\alpha_k\not\in \Delta(P)$ if and only if $\omega_k(\kappa(\lambda))=0$ for all such $k$, i.e., $\kappa(\lambda)$ is a linear combination of $\alpha_i^{\vee}$ with $i\in \Delta(P)$, i.e., is in $\frh_{L^{\op{ss}}}$.
\end{proof}
\subsection{ Proof of Theorem \ref{general_Induction}}

Denote the right hand side of \eqref{dinner} by $(\nu_1,\dots,\nu_s)$. We divide the proof into steps:
 \begin{enumerate}
 \item We first verify that $(\nu_1,\dots,\nu_s)$ is indeed in $\Pic(\Fl'_G)$, since a priori it is only in $\Pic(\Fl_G)$. Consider a pair $(j,v)$ with $v\leto{\alpha_{\ell}}w_j$. We want $\nu_j(\alpha_{\ell}^{\vee})=0$. This vanishing follows immediately from Corollary \ref{correspondence2}.
 \item Next, we verify that the line bundle $\mathcal{N}$ given by $(\nu_1,\dots,\nu_s)$ agrees with $\op{Ind}(\ml)$, on $\mathcal{C}-\mathcal{R}$, considered an open subset of $\Fl_G$  where $\ml$ is the line bundle on $\Fl_L$ given by $(\mu_1,\dots,\mu_s)$. To do this we only need to show, by Lemma \ref{comparison} that the pullbacks to $\Fl_L-\mathcal{R}_L$ via $i$ agree. By Lemma \ref{malta}, the line bundles  $\mathcal{O}(D(j,v))$ pull back to trivial line bundles on $\Fl_L-\mathcal{R}_L$. It is now easy to verify the desired agreement.
 \item Therefore one has a relation $\mathcal{N}=\op{Ind}(\ml)(D)$ on  $\Fl_G$ with $D$ supported on the complement of the image of $\mathcal{C}-\mathcal{R}$, i.e., a sum of divisors $D(j,v)$.
\item But such a sum of divisors $D$ needs to be zero
because both sides are in $\Pic(\Fl'_G)$, and Corollary \ref{correspondence2}. This concludes the proof of Theorem \ref{general_Induction}.
\end{enumerate}
We note that steps (2) and (3) are very similar to Ressayre's proof \cite{R1} of the irredundancy of inequalities \eqref{kishi} for maximal parabolics $P$. The divisors analogous to $D$ in op. cit. are not determined. We are able to determine it (as zero) because of the enumerative computations of Theorem \eqref{timeticks} (as in Corollary \ref{correspondence2}).
\begin{corollary}\label{apples}
Let $\alpha_k\not\in\Delta(P)$, then induction of the various
$(0,\dots,0,\omega_k,0,\dots,0)$ with $\omega_k$ in the $j$th place coincide, i.e., the following elements of
$\Pic(\Fl'_G)\subseteq \Pic(\Fl_G)$ are the same:
$$(0,\dots,0,w_j\omega_k,0,\dots,0) -\sum_{v} w_j\omega_k(\alpha_{\ell}^{\vee})\mathcal{O}(D(j,v)).$$
Moreover, this element fails the inequality defining $\mathcal{F}$ and is thus not a member of $\Gamma(s,K)$.
\end{corollary}

\begin{proof}
This last claim is seen from the following: each $\mathcal{O}(D(j,v))$ appearing in the sum vanishes on the inequality for $\mathcal{F}$ by Theorem \ref{Tone}(d), and
$$
w_j^{-1}(w_j\omega_k)(x_k)=\omega_k(x_k)>0,
$$
since $\omega_k(x_k) = c(\omega_k,\omega_k)$ for some $c>0$.
\end{proof}

\begin{remark}
The quantity $w_j\mu_j(\alpha_{\ell}^{\vee})$ in \eqref{dinner} is $\leq 0$ if $\mu_j$ is dominant:
$$w_j\mu_j(\alpha_{\ell}^{\vee})= \frac{2}{(\alpha_{\ell},\alpha_{\ell})}(w_j\mu_j,\alpha_{\ell})= \frac{2}{(\alpha_{\ell},\alpha_{\ell})}(\mu_j, w_j^{-1}\alpha_{\ell})$$
but $w_j^{-1}\alpha_\ell\in R^-$.

If $(\mu_1,\dots,\mu_s)$ is a tuple of  dominant weights for $G$ then the $\lambda_j$ in \eqref{dinner} are also dominant: We
compute $\lambda_j(\alpha_m^{\vee})$: If $w_j\mu_j(\alpha_m^{\vee})\geq 0$, then there is nothing to
show, given the non-negativity from previous paragraph. If $w_j\mu_j(\alpha_m^{\vee})<0$ then clearly
$w^{-1}\alpha_m\in R^-$, and the divisor $D(j,v)$ with $v=s_{\alpha_m}w_j$ appears in the sum, and hence
$\lambda_j(\alpha_m^{\vee})=0.$

We note however that an element in $\Pic^+(\Fl_L)$ may not necessarily be representable by a tuple $(\mu_1,\dots,\mu_s)$ of dominant weights for $G$, and it is therefore a consequence of our results that
the formulas for induction of elements in $\Pic^+(\Fl_L)$ produce  tuples of dominant weights $(\lambda_1,\dots,\lambda_s)$ for $G$ in \eqref{dinner}.
\end{remark}
\begin{remark}
By \cite[Section 3]{BK}, $\mathcal{O}(\mathcal{R}_L)$ is the line bundle on $\Fl_L$ given by (see Lemma \ref{lemmE}(2)) the $s$-tuple $(\chi_{w_1}-\chi_{e},\chi_{w_2},\dots,\chi_{w_s})$ where
$\chi_w=\rho-2\rho^L+ w^{-1}\rho$. Here $\rho$ (resp. $\rho^L$) is the half sum of roots in $R^+$ (resp. in $R^+_{\mathfrak{l}}$). Since the induction of $\mathcal{O}(\mathcal{R}_L)$ is zero, we get an interesting relation  from \eqref{dinner}.

\end{remark}

\section{Related results}\label{compl}
Let $P$ and the $w_j$ be as earlier, and let
\begin{equation}
\mathcal{H}=\{(h_1,\dots,h_s)\in \frh_{\Bbb{Q}}^s\mid \sum_{j=1}^s \omega_{{k}}(w_j ^{-1}h_j) = 0,\  \alpha_{k}\not\in \Delta(P)\}\subseteq \frh_{\Bbb{Q}}^s.
\end{equation}
\begin{lemma}
$\mathcal{H}\subseteq \frh_{\Bbb{Q}}^s$ is of codimension $|\Delta-\Delta(P)|$, and not contained in any root hyperplane $\alpha_m(h_j)=0$.
\end{lemma}
\begin{proof}
We need to show that these equations are linearly independent: If fewer equations cut out the same set, it will also be the case if we restrict to $h_2=\dots=h_s=0$. Therefore the equations $\omega_k(h)=0$ where $\alpha_{k}\not\in \Delta(P)$ are linearly dependent which is clearly false.  Finally, $(\frh_{\Bbb{Q}})^s_0$ is not contained in any root hyperplane because it contains points of the form $(w_1 h,-w_2h,\dots,0)$, $h$ arbitrary.
\end{proof}

Clearly $\mf_{\Bbb{Q}}\subseteq \mathcal{H}$, and we will show that it generates it as a vector space. This will show that $\mf_{\Bbb{Q}}$ is a regular face of codimension $|\Delta-\Delta(P)|$, a
result first proved by  Ressayre \cite{R2}.

 To do this we define
\begin{equation}
\mathcal{H}[2]=\{(h_1,\dots,h_s)\in \mathcal{H}\mid  \beta(h_j)=0, \forall  (j,v),
 v\leto{\beta}w_j, \beta\in\Delta, v\in W^P.\}
 \end{equation}

Clearly $\mf_{2,\Bbb{Q}}=\mathcal{H}[2]\cap \mf_{\Bbb{Q}} =\Pic^{+,\deg =0}_{\Bbb{Q}}(\Fl'_G)$, and parallel to $\mf_{\Bbb{Q}}=\Bbb{Q}_{\geq 0}^q \times \mf_{2,\Bbb{Q}}$
we have a decomposition $\mathcal{H}= \Bbb{Q}^q\times  \mathcal{H}[2]$, therefore it suffices to show that $\mf_{2,\Bbb{Q}}$ generates $\mathcal{H}[2]$. Therefore we need to show that $\Pic^{\deg =0}_{\Bbb{Q}}(\Fl'_G)$ is generated by $\Pic^{+,\deg =0}_{\Bbb{Q}}(\Fl'_G)$.

There are surjections (Lemmas \ref{ones} and \ref{twos}),
$$\Pic_{\Bbb{Q}}^{\deg=0}(\Fl_L)\twoheadrightarrow \Pic^{\deg =0}_{\Bbb{Q}}(\Fl'_G),\ \ \ \Pic_{\Bbb{Q}}^{+,\deg=0}(\Fl_L)\twoheadrightarrow \Pic^{+,\deg =0}_{\Bbb{Q}}(\Fl'_G),$$
and identifications (Lemma \ref{threes})
$$\Pic_{\Bbb{Q}}^{+,\deg=0}(\Fl_L)=\Pic_{\Bbb{Q}}^{+}(\Fl_{L^{\op{ss}}}),\ \ \Pic_{\Bbb{Q}}^{\deg=0}(\Fl_L)= \Pic_{\Bbb{Q}}^{\deg=0}(\Fl_{L^{\op{ss}}}).$$

The desired statement that $\Pic_{\Bbb{Q}}^{+}(\Fl_{L^{\op{ss}}})$ generates $ \Pic_{\Bbb{Q}}^{\deg=0}(\Fl_{L^{\op{ss}}})$, now follows by reducing to simple factors, and the known fact that $\Gamma_{\Bbb{Q}}(s,K)$ generates $\frh_{\Bbb{Q}}^s$ ($s\geq 3$) for simple, simply connected $G$.

\subsection{Irreducible components}
Consider the inverse image $\mathcal{R}_{(L/B_L)^s}\subseteq (L/B_L)^s$ of the divisor $\mathcal{R}_L\subseteq \Fl_L$ Let $c$ be the number of irreducible components of $\mathcal{R}_{(L/B_L)^s}$.  Let $\mathcal{R}_1,\dots,\mathcal{R}_c$ be the irreducible (reduced) components of this divisor. It is easy to see that each is left invariant by the connected group $L$, and hence they give line bundles $\mathcal{O}(\mathcal{R}_1),\dots,\mathcal{O}(\mathcal{R}_c)$ on $\Fl_L$. Since
$H^0(\Fl_L-\mathcal{R}_L,\mathcal{O})$ is one dimensional, we have that $H^0(\Fl_L,\otimes_{i=1}^c\mathcal{O}(N_i\mathcal{R}_i))$ is one dimensional for  $N_i>0$, $i=1,\dots,c$.
\begin{lemma}
$\mathcal{O}(\mathcal{R}_1),\dots,\mathcal{O}(\mathcal{R}_c)$ give a $\Bbb{Z}$-basis for the kernel of
 $\Pic(\Fl_L)\to \Pic(\Fl_L-\mathcal{R}_L)$
\end{lemma}
\begin{proof}
It is easy to see that $\mathcal{O}(\mathcal{R}_1),\dots,\mathcal{O}(\mathcal{R}_c)$ are in the kernel. They are linearly independent because any isomorphism of line bundles  $\mathcal{O}(\sum_{i\in I} a_i \mathcal{R}_i)=\mathcal{O}(\sum_{j\in J} b_j \mathcal{R}_j)$, with $I,J$ disjoint and $a_i>0$ and $b_j>0$, produces two linearly independent sections in the the isomorphic line bundles.

They span, because if $\mathcal{L}\in \Pic(\Fl_L)$ maps to zero, then first the action of $Z^0(L)$ is trivial, and hence we have to show that $\mathcal{L}$ is isomorphic to a linear combination of the pull backs of the line bundles  $\mathcal{O}(\mathcal{R}_i)$, when pulled back to $(L/B_L)^s$, without any equivariance conditions. Let $s$ be a non-zero section of $\mathcal{L}$ on  $\Fl_L-\mathcal{R}_L$; clearly the pull back of $s$ to $(L/B_L)^s$ has associated divisor supported on the union of $\mathcal{R}_i\subseteq (L/B_L)^s$ which completes the argument.
\end{proof}
\begin{proposition}\label{countem}
$c=q-(s-1)|\Delta-\Delta(P)|$. Recall that $q$ is the number of divisors $D(j,v)$.
\end{proposition}
\begin{proof}
We count dimensions in the isomorphism $\Pic^{\deg=0}_{\Bbb{Q}}(\Fl_L-\mathcal{R}_L)=\Pic^{\deg=0}_{\Bbb{Q}}(\Fl'_G)$. The right hand side has dimension $\dim \mh -q= \dim(G/B)^s-|\Delta-\Delta(P)| -q$.

Using the surjection $\Pic^{deg=0}(\Fl_L)\to \Pic^{\deg=0}(\Fl_L-\mathcal{R}_L)$, we see that the left hand side has dimension equal to $\dim \Pic^{\deg=0}_{\Bbb{Q}}(\Fl_L)-c=\dim(G/B)^s -s|\Delta-\Delta(P)|-c$  (see Lemma \ref{twos}(2)).
The result follows.
\end{proof}
 \begin{lemma}
$\mathcal{O}(\mathcal{R}_1),\dots,\mathcal{O}(\mathcal{R}_c)$ give (some) extremal rays of $\Pic_{\Bbb{Q}}^{+,\deg=0}(\Fl_L)$.
\end{lemma}
\begin{proof}
This follows the same method of proof as of Theorem \ref{Tone}, (b) implies (c), using the fact noted above that
$H^0(\Fl_L,\mathcal{O}(N\mathcal{R}_i))$ is one dimensional if $N\geq 0$ for any $i$..
\end{proof}
\subsection{The face $\mf=\mf(\vec w,P)$ when $P=B$}
Clearly $L^{\op{ss}}=\{e\}$ when $P=B$ and hence $\mf_{2,\Bbb{Q}}=0$, and $\mf_{\Bbb{Q}}$ is the
cone spanned by the linearly independent $\delta_1,\dots,\delta_q$. Therefore the dimension of $\mf_{\Bbb{Q}}$ is $q$, while at the same time it is $sr-r=(s-1)r$, $r=|\Delta|$, and $c=0$.
\begin{corollary}
The regular faces of $\Gamma(s,K)$ of codimension $|\Delta|$ (the maximum possible) are simplicial cones.
\end{corollary}
\begin{remark}\label{especial}
The following types of surjection of cones $\mathcal{C}\twoheadrightarrow \overline{\mathcal{C}}$ can be considered special: Let $V=\Bbb{Q}^n$, and $\mathcal{C}\subseteq V$ a (spanning) cone that has the basis vectors $e_1,\dots,e_c$ among its extremal rays. Let  $V\to \Bbb{Q}^{n-c}$ be the projection to the remaining $n-c$ coordinates. Let $\overline{\mathcal{C}}\subseteq \Bbb{Q}^{n-c}$ be the image of $\mathcal{C}$.

The surjection of cones \eqref{indsurj} is of the above special type (take $V=\Pic_{\Bbb{Q}}^{\deg=0}(\Fl_L)$, and $e_1,\dots,e_c$ the elements $\mathcal{O}(\mathcal{R}_1),\dots,\mathcal{O}(\mathcal{R}_c)$, using the bijection  of Proposition \ref{bij}, and Lemma \ref{threes}).

Under the bijection of Proposition \ref{bij}, the surjection of cones \eqref{indsurj}  becomes
$\Pic_{\Bbb{Q}}^{+,\deg=0}(\Fl_L)\twoheadrightarrow \Pic_{\Bbb{Q}}^{+, \deg=0}(\Fl'_G)$. We note that
this has a section arising from Lemma \ref{twos}.

\end{remark}

%\appendix
\section{Examples}\label{app1}

In the following, we examine several facets of the $D_4$ tensor cone ($s=3$), producing type I and type II rays according to the formulas given earlier. All rays produced here can also be found in the (complete up to symmetrization) list of $81$ extremal rays for $D_4$ in \cite{KKM}. In fact, all $81$ extremal rays are type I on {\it some} face. Type I extremal rays, under the bijection of Proposition \ref{bij}, have the property that any multiple has an exactly  one dimensional space of invariant  global sections, see  Theorem \ref{Tone}, (b).  There are examples in type $A$ in \cite{BHermit}, due to Derksen-Weyman \cite[Example 7.13]{DW} and Ressayre, of extremal rays for $\op{SL}(8)$ and  $\op{SL}(9)$ respectively,  which do not have  this property, and  give  examples of   extremal rays which are not type I on any face. There are similar examples which do not have this property for $D_5$ in \cite{Kiers}. 

Some rudimentary computer code, written using the free math software {\tt Sage}, was used to find tuples $(u,v,w,P)$ giving rise to facets and to implement the formulas found in Theorem \ref{timeticks} and Theorem \ref{general_Induction}.

The details of these computer algorithms will appear in \cite{Kiers}.

\subsection{A face coming from $P_2$}\label{subbie}
Let $G$ be of type $D_4$, with simple roots $\alpha_1,\alpha_2,\alpha_3,\alpha_4$ and corresponding simple reflections $s_i$. Let $P=P_2$ and $u,v,w$ be specified as in the example in \ref{ex1}. On the corresponding face $\mathcal{F}$, there are $7$ type I extremal rays, generated by:
$$
\begin{array}{c}
(\omega_1,\omega_4,\omega_3),\\
(\omega_3,0,\omega_3)\text{ and }(\omega_4,\omega_4,0),\\
(\omega_2,\omega_3,\omega_3)\text{ and }(\omega_2,\omega_4,\omega_4),\\
(\omega_2,\omega_1+\omega_4,\omega_3)\text{ and }(\omega_2,\omega_4,\omega_1+\omega_3).
\end{array}
$$

One may note that under the operation of switching entries $2$ and $3$ (i.e., $(u,v,w)$ becomes $(u,w,v)$, $(\lambda,\mu,\nu)$ becomes $(\lambda,\nu,\mu)$) while simultaneously switching indices $3$ and $4$ (on all simple roots, fundamental dominant weights, simple reflections; this is a Dynkin diagram automorphism), the specific $(u,v,w,P)$ of the example remains unchanged. Therefore the face $\mathcal{F}$ is also invariant under the induced cone automorphism; the above type I rays are listed in pairs according to this (order 2) automorphism (the first is fixed).

The induction map gives the following four type II rays:
$$
\begin{array}{c}
(\omega_2,\omega_2,0)\text{ and }(\omega_2,0,\omega_2),\\
(\omega_2,\omega_2,2\omega_3)\text{ and }(\omega_2,2\omega_4,\omega_2),
\end{array}
$$
again given in pairs. The Levi associated to $P$ is of type $A_1\times A_1\times A_1$. The tensor cone for type $A_1$ is generated (over $\mathbb{Z}$ as well as over $\mathbb{Q}$) by three extremal rays: $(\omega,\omega,0)$ and its two permutations, where $\omega$ is the single dominant fundamental weight. The dominant fundamental weights for $L$ are $\omega_1$, $\omega_3$, and $\omega_4$, each representing a copy of $A_1$. Extremal rays for the type $A_1\times A_1\times A_1$ subcone are therefore given by the permutations of $(\omega_i,\omega_i,0)$, where $i$ runs through $1,3,4$, yielding a total of $9$.

These $9$ rays are shifted by a multiple of $\omega_2$ in each entry so that the result evaluates to $0$ against $x_2$ (in each entry); i.e., each ray is shifted to become degree $0$ (see Lemma \ref{lisse}). The formula for induction (\ref{dinner}) is then applied, with the following results:

$$
\begin{array}{cclcccl}
(\omega_1,\omega_1,0) & \mapsto & \vec 0& &
(\omega_3,\omega_3,0) & \mapsto & \vec 0\\
(\omega_4,\omega_4,0) & \mapsto & (\omega_2,\omega_2,0)&&
(0,\omega_1,\omega_1) & \mapsto & \vec 0\\
(0,\omega_3,\omega_3) & \mapsto & (\omega_2,2\omega_4,\omega_2)&&
(0,\omega_4,\omega_4) & \mapsto & (\omega_2,\omega_2,2\omega_3)\\
(\omega_1,0,\omega_1) & \mapsto & \vec 0&&
(\omega_3,0,\omega_3) & \mapsto & (\omega_2,0,\omega_2)\\
(\omega_4,0,\omega_4) & \mapsto & \vec 0. &&&&
\end{array}
$$

These $11$ rays are indeed all of the extremal rays on $\mathcal{F}$. Notice that $c = \#\text{ irreducible components of }\mathcal{R}_L = 5$, the number of extremal rays going to $0$ under induction. Here $q = 7$, $s=3$, and $|\Delta - \Delta(P)|=1$, so $5=7-(3-1)(1)$ illustrates Proposition \ref{countem}.

Finally, in this example, any extremal ray for $L$ which does not go to $\vec 0$ is induced to a type II ray. This is not always the case:

\subsection{Illustration of Corollary \ref{apples}} Maintaining $P=P_2$ and $u,v,w$ as above, we examine the induction operation (without any shifting) applied to $(u\cdot\omega_2,0,0)$, $(0,v\cdot\omega_2,0)$, and $(0,0,w\cdot\omega_2)$.

First $(u\cdot\omega_2,0,0)$: one may check that $u\cdot\omega_2 = 2\omega_2 - \omega_1 - \omega_3 - \omega_4$;
$$
s_4s_3s_1s_2(\epsilon_1+\epsilon_2) = s_4s_3s_1(\epsilon_1+\epsilon_3) = s_4s_3(\epsilon_2+\epsilon_3) = s_4(\epsilon_2+\epsilon_4) = \epsilon_2-\epsilon_3,
$$
and indeed
\begin{align*}
2\omega_2-\omega_1-\omega_3-\omega_4 &= 2(\epsilon_1+\epsilon_2) - \epsilon_1- \frac{1}{2}(\epsilon_1+\epsilon_2+\epsilon_3-\epsilon_4) - \frac{1}{2}(\epsilon_1+\epsilon_2+\epsilon_3+\epsilon_4)\\ &= \epsilon_2-\epsilon_3.
\end{align*}

The type I rays and coefficients coming from divisors $D(j,v)$ for $j=1$ are:
\begin{center}
\begin{tabular}{c|c|c}
$\ell$ & $\mathcal{O}(D(j,v))$ & $u\cdot\omega_2(\alpha_\ell^\vee)$ \\\hline
$1$ & $(\omega_1,\omega_4,\omega_3)$ & $-1$\\
$3$ & $(\omega_3,0,\omega_3)$ & $-1$\\
$4$ & $(\omega_4,\omega_4,0)$ & $-1$
\end{tabular}
\end{center}

Therefore $(u\cdot\omega_2,0,0)$ is mapped to
\begin{align*}
(2\omega_2-\omega_1-\omega_3-\omega_4,0,0)+(\omega_1,\omega_4,\omega_3)+(\omega_3,0,\omega_3)+(\omega_4,\omega_4,0)&=(2\omega_2,2\omega_4,2\omega_3),
\end{align*}
and one may check that $\left(u^{-1}\cdot 2\omega_2 + v^{-1}\cdot 2\omega_4 + w^{-1}\cdot 2\omega_3\right)(x_2) = 2 \not\le 0$, so this induced triple is not in the cone.

Second $(0,v\cdot \omega_2,0)$: $v\cdot \omega_2 = -\omega_1-\omega_3+\omega_4$. The type I rays and coefficients coming from divisors $D(j,v)$ with $j=2$ are
\begin{center}
\begin{tabular}{c|c|c}
$\ell$ & $\mathcal{O}(D(j,v))$ & $u\cdot\omega_2(\alpha_\ell^\vee)$ \\\hline
$1$ & $(\omega_2,\omega_1+\omega_4,\omega_3)$ & $-1$\\
$3$ & $(\omega_2,\omega_3,\omega_3)$ & $-1$
\end{tabular}
\end{center}

Therefore $(0,v\cdot \omega_2,0)$ is mapped to
$$
(0, -\omega_1-\omega_3+\omega_4,0) + (\omega_2,\omega_1+\omega_4,\omega_3) + (\omega_2,\omega_3,\omega_3) = (2\omega_2,2\omega_4,2\omega_3)
$$
as well.

Finally $w\cdot\omega_2 = -\omega_1+\omega_3-\omega_4$. The type I rays and coefficients coming from divisors $D(j,v)$ with $j=3$ are
\begin{center}
\begin{tabular}{c|c|c}
$\ell$ & $\mathcal{O}(D(j,v))$ & $u\cdot\omega_2(\alpha_\ell^\vee)$ \\\hline
$1$ & $(\omega_2,\omega_4,\omega_1+\omega_3)$ & $-1$\\
$4$ & $(\omega_2,\omega_4,\omega_4)$ & $-1$
\end{tabular}
\end{center}

Therefore $(0,v\cdot \omega_2,0)$ is mapped to
$$
(0,0, -\omega_1+\omega_3-\omega_4) + (\omega_2,\omega_4,\omega_1+\omega_3) + (\omega_2,\omega_4,\omega_4) = (2\omega_2,2\omega_4,2\omega_3)
$$
yet again.

\subsection{The faces coming from $P_4$}
Take $P=P_4$. The Levi associated to $P$ is of type $A_3$, whose tensor cone is generated by $18$ extremal rays: $(\omega_1,\omega_3,0)$, $(\omega_2,\omega_2,0)$, $(\omega_2,\omega_3,\omega_3)$, $(\omega_2,\omega_2,\omega_1+\omega_3)$, $(\omega_2,\omega_1,\omega_1)$, and permutations: To apply induction, we want to get them in $\deg=0$ part of $\Pic(FL_L)$ (as in Section \ref{subbie}). We shift each entry by a multiple of $\omega_4$  so that the result evaluates to $0$ against $x_4$ (in each entry), see Lemma \ref{lisse}.

It is possible for an induced ray to be non-zero and non-extremal (call such a ray ``exotic''); this happens on several faces arising from $P_4$. For instance, on the face $\mf(s_2s_4,s_3s_1s_2s_4,s_4s_2s_3s_1s_2s_4,P_4)$, the extremal ray $(\omega_2,\omega_2,\omega_1+\omega_3)$ for $A_3$ is induced to $(\omega_1+\omega_3+\omega_4,\omega_2+\omega_4,\omega_1+\omega_3)$, which is not an extremal ray for $\mf$ because it can be expressed as the sum of two distinct extremal rays of $\mf$:
$$
(\omega_1+\omega_3+\omega_4,\omega_2+\omega_4,\omega_1+\omega_3) = (\omega_1+\omega_4,\omega_2,\omega_3) + (\omega_3,\omega_4,\omega_1).
$$

The following table summarizes some characteristics of the $7$ faces (up to symmetrization) coming from $P_4$:
\\
\begin{center}
\begin{tabular}{|c|c|c|c|c|}\hline
Weyl triple & $q$ & $c$ & exotic induced rays & total rays $(q+18-c-e)$\\\hline\hline
$(1,s_4s_2s_3s_1s_2s_4,s_4s_2s_3s_1s_2s_4)$ & 2 & 0 & none & 20 \\\hline
$(s_4,s_2s_3s_1s_2s_4,s_4s_2s_3s_1s_2s_4)$ & 3 & 1 & none & 20 \\\hline
$(s_2s_4,s_3s_1s_2s_4,s_4s_2s_3s_1s_2s_4)$ & 4 & 2 & 1 & 19 \\\hline
$(s_2s_4,s_2s_3s_1s_2s_4,s_2s_3s_1s_2s_4)$ & 3 & 1 & 6 & 14 \\\hline
$(s_3s_2s_4,s_1s_2s_4,s_4s_2s_3s_1s_2s_4)$ & 3 & 1 & 1 & 19 \\\hline
$(s_3s_2s_4,s_3s_1s_2s_4,s_2s_3s_1s_2s_4)$ & 4 & 2 & 3 & 17 \\\hline
$(s_1s_2s_4,s_3s_1s_2s_4,s_2s_3s_1s_2s_4)$ & 4 & 2 & 3 & 17 \\\hline
\end{tabular}
\end{center}
~\\

\vspace{0.05 in}

\noindent
Department of Mathematics, University of North Carolina, Chapel Hill, NC 27599\\
{{email: belkale@email.unc.edu (PB),  jokiers@live.unc.edu (JK)}}

\end{document}